\documentclass{birkjour}
 \newtheorem{thm}{Theorem}[section]
 \newtheorem{cor}[thm]{Corollary}
 \newtheorem{lem}[thm]{Lemma}
 
 \theoremstyle{definition}
 
 \theoremstyle{remark}
 \newtheorem{rem}[thm]{Remark}
 
 \numberwithin{equation}{section}

\usepackage{amsmath, amsthm, amscd, amsfonts, amssymb, graphicx, color}
\usepackage{latexsym}
\usepackage{verbatim}
\usepackage{amscd}
\usepackage{amsthm}
\usepackage{amssymb}
\usepackage{enumerate}
\usepackage{mathabx}
\usepackage{color}
\usepackage{soul}
\setstcolor{red}
\setul{}{0.4ex}

\begin{document}

%
%
%
%
%
%
%
%
%

\title[Lions Problem for Quasi-Banach Spaces]
 {Using the Baire Category Theorem to \\ Explore Lions Problem for \\ Quasi-Banach Spaces.}


\author[A.~G.~Aksoy]{A.~G.~Aksoy}

\address{%
Department of mathematics\\
Claremont McKenna College\\
Claremont, CA, 91711, USA.}

\email{aaksoy@cmc.edu}

\author{J. M. Almira}
\address{Depto. Ingenier\'{\i}a y Tecnolog\'{\i}a de Computadores,  Universidad de Murcia. \\
 30100 Murcia, SPAIN\\}
\email{jmalmira@um.es}


\date{October 23, 2024}
\dedicatory{Dedicated to the memory of Prof. A. Pietsch, the master of us all.}

\begin{abstract}
Many results for Banach spaces also hold for quasi-Banach spaces. One important such example is results depending on the Baire Category Theorem (BCT). We use the BCT to explore Lions problem for a quasi-Banach couple $(A_0, A_1)$. Lions problem, posed in 1960's, is to prove that different parameters $(\theta,p)$ produce different interpolation spaces $(A_0, A_1)_{\theta, p}$. We first establish conditions on $A_0$ and $A_1$ so that interpolation spaces of this couple are strictly intermediate spaces  between $A_0+A_1$ and $A_0\cap A_1$.  This result, together with a reiteration theorem, gives a partial solution to Lions problem for quasi-Banach couples. We then apply our interpolation result  to (partially) answer a question posed by Pietsch.  More precisely, we show that if $p\neq p^*$ the operator ideals $\mathcal{L}^{(a)}_{p,q}(X,Y)$, $\mathcal{L}^{(a)}_{p^*,q^*}(X,Y)$ generated by approximation numbers are distinct. Moreover, for any fixed $p$, either all operator ideals $\mathcal{L}^{(a)}_{p,q}(X,Y)$ collapse into a unique space or they are pairwise distinct.   We cite counterexamples which show that 
using interpolation spaces is not appropriate to solve Pietsch's problem for operator ideals based on general $s$-numbers. However, the BCT can be used to prove a lethargy result for arbitrary $s$-numbers which guarantees that, under very minimal conditions on $X,Y$,  the space $\mathcal{L}^{(s)}_{p,q}(X,Y)$ is strictly embedded into $\mathcal{L}^{\mathcal{A}}(X,Y)$. 

\end{abstract}

\maketitle

\section{Introduction}

In the following we give a brief review of notation, terminology and the background relevant to this paper. If $M$ and $N$ are non-negative quantities, we write $M\lesssim N$ if there is a constant $c>0$ independent of the parameters involved in $M$ and $N$ such that $M \leq c \, N$. The symbol $\asymp$ stand for asymptotic equivalence, and $M\asymp N$ means that $M\lesssim N$ and $N\lesssim M$. We say $(A_0,A_1)$ is a quasi-Banach couple if $A_0$ and $A_1$ are quasi-Banach spaces both continuously embedded in some Hausdorff topological vector space. Peetre's K-functional of a quasi-Banach couple $(A_0,A_1)$ is defined by
$$K(x,t,A_0,A_1):=\inf \{||a_0||_{A_0}+t||a_1||_{A_1}:\,\, x=a_0+a_1, \,\, a_j\in A_j, \,\,  j=0,1 \}$$ 
for $t>0$ and $x\in A_0+A_1$. Note that the quasi-norm on the intersection $A_0 \cap A_1$ is 
$||x||_{A_0\cap A_1} :=\max \{||x||_{A_0}, ||x||_{A_1}\}\asymp ||x||_{A_0}+||x||_{A_1}$  while on the sum is given by  $||x||_{A_0+A_1} =  K(x,1,A_0,A_1)$.

For $0 < \theta < 1$ and $0 < p \leq \infty$ (or for $\theta = 0, 1$ with $p = \infty$), one defines the \textit{real interpolation space} as follows: 
\[(A_{0}, A_{1})_{\theta, p} := \left\{x \in A_{0} + A_{1} : \quad t^{-\theta} K(x, t) \in L_{p} \left([0, \infty), \frac{dt}{t}\right) \right \} \] with the quasi-norm 
$||f||_{(A_{0}, A_{1} )_{\theta, p}}: =  \begin{cases} \left( \int_{0}^{\infty} \left(t^{-\theta} K(x, t) \right)^{p} \frac{dt}{t} \right)^{\frac{1}{p}}, \quad  0 < p < \infty \\ \sup_{0 <t < \infty} t^{-\theta} K(x, t), \quad\quad   p = \infty. \end{cases}$\\

If $(A_0,A_1)$ is a quasi-Banach couple and $A$ is  a subset of $A_0+A_1$, then we will also use the following notation:
$$K(A,t,A_0,A_1)=\sup_{a\in A} K(a,t,A_0,A_1).$$ 
As usual, $S(X)$ denotes the unit sphere of the quasi-Banach space $X$. In this paper, the function  $K(S(A_0+A_1),t,A_0,A_1)$ will be important.

When dealing with quasi-Banach spaces $X$, some limitations appear. Concretely, Hahn-Banach theorem does not hold, and duality becomes intractable,  and because of the weakness of the triangle inequality convexity arguments do not work well. However Aoki-Rolewicz theorem offers an alternative since it defines an equivalent quasi-norm $|\, .\, |$ on $X$ for which $|x+y|^p \leq |x|^p+|y|^p$ for certain $0<p<1$ \cite{DvL}. Moreover, many results of Banach spaces still hold true, specially those depending on Baire Category Theorem (BCT).

A well known question due to J.L. Lions is how to prove that a given family of interpolation spaces strictly depend on their parameters. This question was answered by Janson, Nilsson, Peetre and Zafran in \cite{JNPZ} for the real method. They showed that  for any given Banach couple $(A_0,A_1)$, if $1\leq p,q \leq \infty$,  $0<\theta,  \eta < 1$  and $A_0\cap A_1 $ is not closed in $A_0+A_1$, then  $$(A_0,A_1)_{\theta,p}=(A_0,A_1)_{\eta,q}\quad\mbox{implies that}\quad \theta=\eta, \,\,\mbox{and}\,\, p=q.$$ 
A different approach to the Lions problem in the case of ordered quasi-Banach couples was given in \cite{AF} where the authors characterized ordered couples $(A_0,A_1)$ whose associated K-functional slowly decays to zero. This result was utilized to show that some natural embeddings between real interpolation spaces are strict, and they got, for the first time in quasi-Banach setting but under the additional restriction of assuming ordered couples, a positive solution to all cases of Lions problem except one. Later, in \cite{CK} (see also \cite{CK1}), Cobos, Cwikel, and K\"{u}hn, solved, with more sophisticated tools,  the general case of Lions problem in quasi-Banach setting, under the additional hypothesis that $(A_0,A_1)$ is a Gagliardo couple, and they used their result to demonstrate that the operator ideals $\mathcal{L}_{p,q}^{(a)}(X,Y)$ are pairwise distinct for arbitrary Banach spaces $X,Y$.

In section \ref{secdos} of this paper we partially extend the results of \cite{AF} to arbitrary quasi-Banach couples. In section 3 we observe that the results of section \ref{secdos} can be applied to the study of the ideals of operators based on approximation numbers, with no necessity of checking that  $(\mathcal{L}_{p_0,p_0}^{(a)}(X,Y),\mathcal{L}_{p_1,p_1}^{(a)}(X,Y))_{\theta,q}$  is a Gagliardo couple. We also observe that, if $X$ has no nontrivial type and $Y$ has no nontrivial cotype, then the spaces $\mathcal{L}_{p,q}^{(c)}(X,Y)$, where $c_n(T)$ is the sequence of Gelfand numbers, are pairwise distinct. We also explain that using interpolation spaces is not adequate to study the analogous of Lions problem for general operator ideals $\mathcal{L}_{p,q}^{(s)}(X,Y)$ where $s_n(T)$ denotes an $s$-numbers sequence. Then we use BCT to prove a lethargy result for general $s$-numbers and we apply it in several concrete cases. The paper ends with an appendix where some results about the interpolation spaces defined by the sum $A_0+A_1$ and the intersection $A_0\cap A_1$  are proven for quasi-Banach couples $(A_0,A_1)$.
\section{The real method of interpolation in Quasi-Banach setting} \label{secdos}

In their seminal book \cite{BK} Brudnyi and Krugljak studied the real method of interpolation in a very general form, introducing the concept of parameter space for the $K$-method.  Given the measure space $\Omega=(\mathbb{R}^+,\mu)$, whose measure is given by $\mu(A)=\int_{A}\frac{dt}{t}$ and integration is taken in Lebesgue sense,  a quasi-Banach lattice $\Phi$ over $\Omega$ is named a parameter for the $K$-method (a parameter, in all what follows) if $\min(1,t)\in \Phi$. Given a parameter $\Phi$, and a couple $(A_0,A_1)$ of quasi-Banach spaces, they proved (see \cite[Proposition 3.3.23]{BK}) that $$(A_0,A_1)_{\Phi}=\{a\in A_0+A_1:\|a\|_{(A_0,A_1)_{\Phi}}= \|K(a,t,A_0,A_1)\|_{\Phi}<\infty\}$$ is a quasi-Banach space which is an interpolation space for the couple $(A_0,A_1)$. 
Another proof of the completeness of the space $(A_0, A_1)_{\Phi}$ follows from  \cite[Theorem 1.1]{Mtype} and the ideas developed in \cite[Theorem 3.12]{AL}.  
Moreover, in \cite{BK} the authors proved the following reiteration formula:
\[
((A_0,A_1)_{\Phi_0},(A_0,A_1)_{\Phi_1})_{\Phi}=(A_0,A_1)_{\Sigma},\quad \text{ where } \Sigma=(\hat{\Phi}_0,\hat{\Phi}_1)_{\Phi}.
\]
Here $\hat{\Phi}=(L_{\infty}^0,L_{\infty}^1)_{\Phi}$, where (see \cite[p. 298]{BK}): $$L_{p}^{\alpha}:=L_{p}(t^{-\alpha})(\Omega)=\{f:(0,\infty)\to \mathbb{R}: \|t^{-\alpha}f(t)\|_{L_{p}(\Omega)}<\infty\}, \quad  0<p\leq \infty.$$
They also proved that there exists a discrete version of the spaces $(A_0,A_1)_{\Phi}$. Concretely, for any $q>1$, and any parameter $\Phi$ such that $\Phi\cap \text{\rm{Conv}}\neq\emptyset$, the space
\[
(A_0,A_1)_{\Phi^{(q)}}=\{a\in A_0+A_1: \|\sum_{n\in\mathbb{Z}} K(a,q^n,A_0,A_1)\chi_{[q^n,q^{n+1}]}(t)\|_{\Phi}<\infty\} 
\]
coincides (with equivalent quasi-norms) with the space $(A_0,A_1)_\Phi$. In particular, $(A_0,A_1)_{\Phi^{(q)}}$ is quasi-Banach. Here, $\text{\rm{Conv}}$ denotes set of continuous concave functions $f:(0,\infty)\to [0,\infty)$.  

Thus, if we set $q=2$ and define the sequence space $$S:=S(\Phi)=\{(\alpha_n):\|(\alpha_n)\|_S=\|\sum_{n\in\mathbb{Z}} |\alpha_n|\chi_{[2^n,2^{n+1}]}(t)\|_{\Phi}<\infty\},$$ then 
$$(A_0,A_1)_\Phi= (A_0,A_1)_{S:K}:=\{a\in A_0+A_1: (K(a,2^n,A_0,A_1))\in S\},$$
\[
\|a\|_{(A_0,A_1)_{S:K}}=\|(K(a,2^n,A_0,A_1))\|_S.
\]

These interpolation spaces based on the use of sequence spaces had already been introduced in the literature by by Nilsson in \cite{N}. Given a compatible couple of quasi-Banach spaces $(A_0,A_1)$   and a sequence space $E$ which is a lattice on $\mathbb{Z}$ (this means that there exists $M>0$ such that $\|(\alpha_n)\|_E\leq M \|(\beta_n)\|_E$ whenever $|\alpha_n|\leq |\beta_n|$ for all $n\in\mathbb{Z}$ and $\|(\beta_n)\|_{E}<\infty$), we define $$ (A_0,A_1)_{E:K}=\{ a\in A_0+A_1:\,\,|| (K(2^{\nu},a, A_0,A_1))_{\nu\in \mathbb{Z}}||_E < \infty\}.$$

In \cite{N}, Nilsson computed the $K$-functional for a given pair $\left((A_0,A_1)_{E_0:K}, \right.$ $\left.(A_0,A_1)_{E_1:K}\right)$, where $(E_0,E_1)$ is a couple of interpolation spaces of the couple $(\ell^{\infty}(\mathbb{Z}),\ell^{\infty}((2^{-n})_{n\in\mathbb{Z}}))$ (which is the discrete version of the couple $(L_{\infty}^0,L_{\infty}^1)$), and proved that
 $$K\left(t,a,(A_0,A_1)_{E_0:K},(A_0,A_1)_{E_1:K}\right)\asymp K\left(t, (K(2^{\nu},a, A_0,A_1))_{\nu},  E_0,E_1\right).  $$
 This implies the reiteration theorem:
 $$ \left( (A_0,A_1)_{E_0:K}, (A_0,A_1)_{E_1:K}\right)_{E:K}= (A_0,A_1)_{(E_0,E_1)_{{E:K}}:K}.$$ The classical reiteration theorem of $(A_0,A_1)_{\theta,p}$ (see \cite{BL}, p.51) can be obtained as a special case.

The first two results we present below are about characterization of slowly decaying K-functionals for general quasi-Banach couples $(A_0,A_1)$.

\begin{thm} \label{fundamental}
Assume that $(A_0,A_1)$ is a couple of quasi-Banach spaces, and $A_0$ and $A_1$ are both $p$-normed spaces. Then either $K(S(A_0+A_1),t,A_0,A_1)\geq 1$ for all $t>0$, or $A_0+A_1=A_1$. 
\end{thm}
\begin{proof}
 Assume, on the contrary, that  $K(S(A_0+A_1),t_0,A_0,A_1)=c<1$ and $A_0+A_1\neq A_1$. Take $\rho\in ]0,1[$ such that $c<\rho^{1/p}<1$. Then $$K(\frac{x}{\|x\|_{A_0+A_1}},t_0,A_0,A_1)<\rho^{1/p}$$ for every $x\in A_0+A_1$, $x\neq 0$. Hence every element $x$ of $A_0+A_1$ which is different from zero satisfies $$K(x,t_0,A_0,A_1)<\rho^{1/p}\|x\|_{A_0+A_1},$$
which implies that $x=a_0+b_0$ for certain $a_0\in A_0$ and $b_0\in A_1$ such that  
$\|a_0\|_{A_0}+t_0\|b_0\|_{A_1}<\rho^{1/p}\|x\|_{A_0+A_1}$.  Thus we have that 
\begin{equation} \label{argumento}
\left\{\begin{array}{llllll} x =a_0+b_0 \text{ with } a_0\in A_0, \text{ and } b_0\in A_1 \\
\|a_0\|_{A_0}< \rho^{1/p}\|x\|_{A_0+A_1} \\
\|b_0\|_{A_1}< t_0^{-1}\rho^{1/p}\|x\|_{A_0+A_1} \\
\end{array} \right.
\end{equation}
Let us take $x\in (A_0+A_1)\setminus A_1$ and apply \eqref{argumento} to this concrete element. We can repeat the argument just applying it to $a_0$ (if $a_0=0$ then $x=b_0\in A_1$, which contradicts our assumption). Hence, there are elements $a_1\in A_0$ and $b_1\in A_1$ such that 
\begin{eqnarray*} 
\left\{\begin{array}{llllll} a_0 =a_1+b_1 \text{ with } a_1\in A_0, \text{ and } b_1\in A_1 \\
\|a_1\|_{A_0}< \rho^{1/p}\|a_0\|_{A_0+A_1} \leq \rho^{1/p}\|a_0\|_{A_0}<(\rho^{1/p})^2\|x\|_{A_0+A_1} \\
\|b_1\|_{A_1}< t_0^{-1}\rho^{1/p}\|a_0\|_{A_0+A_1}\leq  t_0^{-1}\rho^{1/p}\|a_0\|_{A_0}< t_0^{-1} (\rho^{1/p})^2\|x\|_{A_0+A_1} \\
\end{array} \right.
\end{eqnarray*}
Moreover, $x=a_0+b_0=a_1+b_1+b_0$. Again $a_1\neq 0$ since $x\not\in A_1$. We can repeat the argument $m$ times to get a decomposition $x=a_m+b_m+\cdots+b_0$ with $a_m\in A_0$, $a_m\neq 0$, $b_k\in A_1$ for all $0\leq k\leq m$ and 
\begin{equation*} 
\left\{\begin{array}{llllll} \|a_m\|_{A_0}< (\rho^{1/p})^{m+1}\|x\|_{A_0+A_1} \\
\|b_k\|_{A_1}< t_0^{-1}(\rho^{1/p})^{k+1}\|x\|_{A_0+A_1} \text{ for all } 0\leq k\leq m .\\
\end{array} \right. 
\end{equation*}
Let us set $z_m=x-a_m=b_0+\cdots+b_m$. Then 
\[
\|x-z_m\|_{A_0+A_1}\leq \|x-z_m\|_{A_0}=\|a_m\|_{A_0}<(\rho^{1/p})^{m+1}\|x\|_{A_0+A_1} \to 0 \text{ for } m\to\infty.
\]
and $x$ is the limit of $z_m$ in the norm of $A_0+A_1$.
On the other hand, if $n>m$, then 
\[
\|z_n-z_m\|_{A_1}^p=\|b_{m+1}+\cdots+b_n\|_{A_1}^p\leq \sum_{k=m+1}^n \|b_k\|_{A_1}^p \leq t_0^{-p}\|x\|_{A_0+A_1}^p\sum_{k=m+1}^n \rho^{k+1},
\]
which converges to $0$ for $n,m\to\infty$. Hence $\{z_m\}$ is a Cauchy sequence in $A_1$ and its limit, $w$, belongs to $A_1$ since $A_1$ is topologically complete. On the other hand, $\|w-z_m\|_{A_0+A_1}\leq \|w-z_m\|_{A_1}\to 0$. This implies $x=w\in A_1$, which contradicts our assumptions. Thus, we have demonstrated, for any couple $(A_0,A_1)$ of $p$-normed quasi-Banach spaces satisfying $A_1\hookrightarrow A_0+A_1$,  that  if $K(S(A_0+A_1),t_0,A_0,A_1)<1$ for a certain $t_0>0$, then $A_0+A_1=A_1$. In particular, if $A_0+A_1\neq A_1$ then $K(S(A_0+A_1),t,A_0,A_1)=1$ for  $0<t\leq 1$ and $K(S(A_0+A_1),t,A_0,A_1)\geq 1$ for $t>1$. 
\end{proof}
\begin{rem}
    Note that $A_0+A_1\neq A_1$ is just a reformulation of $A_1 \hookrightarrow  A_0+A_1$, with strict inclusion. 
\end{rem}
\begin{thm}\label{uno1} Let $(A_0,A_1)$ be a couple of quasi-Banach spaces.
The following are equivalent claims:
\begin{itemize}
\item[$(a)$]  $K(S(A_0+A_1),t,A_0,A_1)>c$ for all  $t>0$ and a certain constant $c>0$.
\item[$(b)$] For every non-increasing sequences $\{\varepsilon_n\},\{t_n\}\in c_0$, there are elements $x\in A_0+A_1$ such that 
\[
K(x,t_n,A_0,A_1)\neq \mathbf{O}(\varepsilon_n)
\]
\end{itemize}
\end{thm}
\begin{proof} $(a)\Rightarrow (b)$.  Let us assume, on the contrary, that $K(x,t_n,A_0,A_1)= \mathbf{O}(\varepsilon_n)$ for all $x\in A_0+A_1$ and certain sequences $\{\varepsilon_n\},\{t_n\}\in c_0$. This can be reformulated as $$A_0+A_1=\bigcup_{m=1}^{\infty} \Gamma_m,$$ where $$\Gamma_m=\{x\in A_0+A_1: K(x,t_n,A_0,A_1)\leq m\varepsilon_n \text{ for all }n\in\mathbb{N}\}.$$
Now, $\Gamma_m$ is a closed subset of $A_0+A_1$ for all $m$ and the Baire category theorem implies that $\Gamma_{m_0}$ has nonempty interior for some $m_0\in\mathbb{N}$. On the other hand, $\Gamma_m=-\Gamma_m$ since $K(x,t,A_0,A_1)=K(-x,t,A_0,A_1)$ for all $x\in A_0+A_1$ and $t> 0$. Furthermore, if $C>1$ is a quasi-norm constant valid for both spaces $A_0$ and $A_1$, then
\[
\textbf{conv}(\Gamma_m)\subseteq \Gamma_{Cm},
\]
since, if $x,y\in \Gamma_m$ and $\lambda\in [0,1]$, then 
\begin{eqnarray*}
&\ & K(\lambda x +(1-\lambda)y ,t_n,A_0,A_1) \\ 
&\ & \quad \quad \leq  C(K(\lambda x ,t_n,A_0,A_1)+ K( (1-\lambda)y ,t_n,A_0,A_1))\\
&\ & \quad \quad =   C(\lambda  K(x ,t_n,A_0,A_1)+(1-\lambda) K( y ,t_n,A_0,A_1)) \\
&\ & \quad \quad \leq    C(\lambda  m\varepsilon_n+(1-\lambda) m\varepsilon_n) \\
&\ & \quad \quad = C m\varepsilon_n 
\end{eqnarray*} 
Thus, if $B_{A_0+A_1}(x_0,r)=\{x\in A_0+A_1:\|x_0-x\|_{A_0+A_1}<r\}\subseteq \Gamma_{m_{0}}$, then $\textbf{conv}(B_{A_0+A_1}(x_0,r)\cup B_{A_0+A_1}(-x_0,r))\subseteq \Gamma_{Cm_0}$. In particular, 
\[
\frac{1}{2}(B_{A_0+A_1}(x_0,r)+B_{A_0+A_1}(-x_0,r))\subseteq \Gamma_{Cm_0}.
\]
Now, it is clear that 
\[
B_{A_0+A_1}(0,r)\subseteq \frac{1}{2}(B_{A_0+A_1}(x_0,r)+B_{A_0+A_1}(-x_0,r)),
\]
since, if $x\in B_{A_0+A_1}(0,r)$, then $\|-x_0+(x+x_0)\|_{A_0+A_1}=\|x_0+(x-x_0)\|_{A_0+A_1}=\|x\|_{A_0+A_1}\leq r$, so that  $x+x_0\in B_{A_0+A_1}(x_0,r)$, $x-x_0\in B_{A_0+A_1}(-x_0,r)$, and $x=\frac{1}{2}((x-x_0)+(x+x_0))$. Thus, 
\[
B_{A_0+A_1}(0,r)\subseteq  \Gamma_{Cm_0}.
\]
This means that, if $x\in (A_0+A_1)\setminus \{0\}$, then 
\[
K(\frac{rx}{\|x\|},t_n,A_0,A_1)\leq Cm_0\varepsilon_n \text{ for all } n=1,2,\cdots .
\]
Hence 
\[
K(x,t_n,A_0,A_1)\leq \frac{\|x\|}{r}Cm_0\varepsilon_n \text{ for all } n=1,2,\cdots .
\]
On the other hand, $(a)$ implies that, for each $n$, there is $x_n\in S(A_0+A_1)$ such that $$K(x_n,t_n,A_0,A_1)>c,$$ so that 
\[
c< K(x_n,t_n,A_0,A_1)\leq  \frac{1}{r}Cm_0\varepsilon_n \text{ for all } n=1,2,\cdots,
\] 
which is impossible, since $\varepsilon_n$ converges to $0$ and $c>0$. This proves $(a)\Rightarrow (b)$. 

Let us demonstrate the other implication. Assume that $(a)$ does not hold. Then there are non-increasing sequences $\{t_n\}, \{c_n\}\in c_0$ such that 
\[
K(S(A_0+A_1),t_n,A_0,A_1)\leq c_n \text{ for all } n\in\mathbb{N}
\]
In particular, if $x\in A_0+A_1$ is not the null vector, then
\[
K(\frac{x}{\|x\|_{A_0+A_1}},t_n,A_0,A_1)\leq K(S(A_0+A_1),t_n,A_0,A_1)\leq c_n \text{ for all } n\in\mathbb{N},
\]
so that 
\[
K(x,t_n,A_0,A_1)=\|x\|_{A_0+A_1} K(\frac{x}{\|x\|_{A_0+A_1}},t_n,A_0,A_1)\leq \|x\|_{A_0+A_1} c_n 
\]
for all $n\in\mathbb{N}$, and $K(x,t_n,A_0,A_1)=\mathbf{O}(c_n)$ for all $x\in A_0+A_1$. This proves $(b)\Rightarrow (a)$. 
\end{proof}

For a quasi-Banach couple $(A_0,A_1)$, our first step is to investigate the conditions on $A_0$ and $A_1$ so that the interpolation spaces of the couple $(A_0,A_1)$ are strictly intermediate spaces between $A_0+A_1$ and $A_0\cap A_1$.


\begin{thm}\label{strict0}
Let $(A_0,A_1)$ be a couple of quasi-Banach spaces and  $E$ be a $\mathbb{Z}$-lattice such that $\{\min(1,2^n)\}_{n\in\mathbb{Z}}\in E$. Then
$$ A_0\cap A_1 \hookrightarrow (A_0,A_1)_{E:K} \hookrightarrow A_0+A_1.$$
Moreover:
\begin{itemize}
\item[$(a)$] Assume that $A_0\neq A_0+A_1$ or $A_1\neq A_0+A_1$. Then 
$$(A_0,A_1)_{E:K} \neq A_0+A_1.$$
\item[$(b)$] Assume that $A_0\cap A_1$ is not closed in $A_0+A_1$. If $\Phi^E\in \ell^{\infty}(2^{\theta n})$ for certain $\theta\in ]0,1[$,  where 
$\Phi^E=\{\|\mathbf{e}_n\|_E\}_{n\in\mathbb{Z}}$, $\mathbf{e}_{k}=\{\delta_{k,n}\}_{n\in\mathbb{Z}}$, $k\in\mathbb{Z}$, then  
$$(A_0,A_1)_{E:K} \neq A_0\cap A_1.$$
\end{itemize}
\end{thm}

\begin{proof} 
 
We prove the first statement for the sake of completeness. The inclusion $A_0\cap A_1\hookrightarrow (A_0,A_1)_{E:K}$  is a direct computation, based on the fact that if $x\in A_0\cap A_1$, then  
\begin{eqnarray*}
K(x,t,A_0,A_1) &=& \inf_{x=a_0+a_1;a_0\in A_0;a_1\in A_1}\|a_0\|_{A_0}+t\|a_1\|_{A_1}\\
& \leq& \min(1,t)\max(\|x\|_{A_0},\|x\|_{A_1}) =\min(1,t)\|x\|_{A_0\cap A_1}
\end{eqnarray*}
By hypothesis, $\{\min(1,2^n)\}_{n\in\mathbb{Z}}\in E$. It follows that, if $a\in A_0\cap A_1$, 
\[
K(a,2^n,A_0,A_1)\leq \min(1,2^n)\|a\|_{A_0\cap A_1} \text{ for all } n\in\mathbb{Z}
\]
Now, $E$ is a $\mathbb{Z}$-lattice and 
$\{\min(1,2^n)\}_{n\in\mathbb{Z}}\in E$ imply that, if we set $
M=\|\{\min(1,2^n)\}_{n\in\mathbb{Z}}\|_E$, then  
\[
\|a\|_{E:K}= \|\{K(a,2^n,A_0,A_1)\}\|_E\leq \|\{\min(1,2^n)\}_{n\in\mathbb{Z}}\|a\|_{A_0\cap A_1}\|_E = M\|a\|_{A_0\cap A_1},
\]
which is what we wanted to prove.

To prove that $(A_0,A_1)_{E:K}\hookrightarrow A_0+A_1$ we proceed as follows: We define $\beta=\{\beta_n\}_{n\in\mathbb{Z}}$ by $\beta_n=0$ for $n<0$ and $\beta_n=1$ for $n\geq 0$. Then we can use that $\|a\|_{A_0+A_1}=K(a,1,A_0,A_1)\leq K(a,2^n,A_0,A_1)$ for all $n\geq 0$, and that $E$ is a $\mathbb{Z}$-lattice to claim that 
\[
\|\|a\|_{A_0+A_1}\beta\|_E\leq \|\{K(a,2^n,A_0,A_1)\}_{n\in\mathbb{Z}}\|_E=\|a\|_{E:K}
\]
so that, taking $C=\frac{1}{\|\beta\|_E}$, 
\[
\|a\|_{A_0+A_1}\leq C\|a\|_{E:K}.
\]

Let us now demonstrate $(a)$. 

Thanks to Aoki-Rolewicz's trick \cite{DvL}, we can assume without loss of generality that $E$ is $r$-normed for some $0<r\leq 1$. Moreover, this renorming does not affect to the other hypotheses we have imposed on $E$. Then we have that $\ell^r(\Phi^E)\subseteq E\subseteq \ell^{\infty}(\Phi^E)$ 
Thus, in order to prove that $(A_0,A_1)_{E:K}$ is strictly contained into $A_0+A_1$ we only need to find an element $a\in A_0+A_1$ such that $\{K(a,2^n,A_0,A_1)\}\not\in \ell^{\infty}(\Phi^E)$.


Indeed, we distinguish two cases:

\noindent \textbf{Case 1:}  $A_1 \hookrightarrow A_0+A_1$ with a strict inclusion. 

Then Theorems \ref{fundamental} and \ref{uno1} guarantee that there are elements $a\in A_0+A_1$ such that $\{K(a,2^{-n},A_0,A_1)\}_{n=0}^{\infty}$ goes to zero as slowly as we want. Thus, we can find $a\in A_0+A_1$ such that $\|\{\|\mathbf{e}_{-n}\|_E K(a,2^{-n},A_0,A_1)\}\|_{\ell_{\infty}(\mathbb{N})}=\infty$, so that  
$$\|\{K(a,2^n,A_0,A_1)\}\|_{\ell^{\infty}(\Phi^E)}=\|\{\|\mathbf{e}_n\|_E K(a,2^{n},A_0,A_1)\}\|_{\ell_{\infty}(\mathbb{Z})}=\infty,$$  
which proves the result.

\noindent \textbf{Case 2:}  $A_0 \hookrightarrow A_0+A_1$ with a strict inclusion.  

Then Theorems \ref{fundamental} and \ref{uno1} again guarantee that there are elements $a\in A_0+A_1$ such that $\{K(a,2^{-n},A_1,A_0)\}_{n=0}^{\infty}$ goes to zero as slowly as we want. Thus, we can find $a\in A_0+A_1$ such that $$\|\{\|\mathbf{e}_n\|_E K(a,2^{n},A_0,A_1)\}\|_{\ell_{\infty}(\mathbb{N})}=\|\{\|\mathbf{e}_n\|_E 2^nK(a,2^{-n},A_1,A_0)\}\|_{\ell_{\infty}(\mathbb{N})}=\infty,$$ so that  
$$\|\{K(a,2^n,A_0,A_1)\}\|_{\ell^{\infty}(\Phi^E)}=\|\{\|\mathbf{e}_n\|_E K(a,2^{n},A_0,A_1)\}\|_{\ell_{\infty}(\mathbb{Z})}=\infty,$$  
which proves the result.





Let us now demonstrate $(b)$. As $A_0\cap A_1 \hookrightarrow A_0+A_1$, we know that $\|y\|_{A_0+A_1}\leq M\|y\|_{A_0\cap A_1}$ for a certain $M>0$ and all $y\in A_0\cap A_1$. 
By hypothesis, $A_0\cap A_1$ is not closed in $A_0+A_1$, so that there exists a sequence $\{y_n\}_{n=0}^{\infty}\subseteq S(A_0\cap A_1)$ such that $\|y_n\|_{A_0+A_1}$ converges to $0$ when $n$ goes to infinity (this is so because the norms $\|\cdot\|_{A_0\cap A_1}$ and $\|\cdot \|_{A_0+A_1}$ are not equivalent). In particular, for every $N_0\in\mathbb{N}$ there exists $y_{N_0}\in S(A_0\cap A_1)$ such that $K(y_{N_0},1,A_0,A_1)=K(y_{N_0},1,A_1,A_0)=\|y_{N_0}\|_{A_0+A_1}\leq 2^{-N_0}$. 

Hence
\begin{eqnarray*}
&\ & \max\{K(y_{N_0},t,A_0,A_1),K(y_{N_0},t,A_1,A_0)\} \\
&\ &  \quad \quad \leq \min(1,t)\|y_{N_0}\|_{A_0\cap A_1}=\min(1,t)
\end{eqnarray*}
and 
\begin{eqnarray*}
&\ &
\max\{K(y_{N_0},t,A_0,A_1),K(y_{N_0},t,A_1,A_0)\}\\
&\ &  \quad \quad\leq \max(1,t)\|y_{N_0}\|_{A_0+A_1}\leq \max(1,t)2^{-N_0}
\end{eqnarray*}
This obviously implies that 
\[
\max\{K(y_{N_0},t,A_0,A_1),K(y_{N_0},t,A_1,A_0)\} \leq \left\{\begin{array}{lllll} 
t & \text{ if } t< \frac{1}{2^{N_0}} \\
\frac{1}{2^{N_0}} & \text{ if } \frac{1}{2^{N_0}}\leq t\leq 1 \\
\\
t\frac{1}{2^{N_0}} & \text{ if } t>1
\end{array} \right.
\]
Now, the inclusion $\ell^r(\Phi^E)\subseteq E$ means that $\|\{\alpha_n\}\|_E\leq C \|\{\alpha_n\}\|_{\ell^r(\Phi^E)}$ for a certain constant $C>0$, and $\Phi^E\in \ell^{\infty}(2^{\theta n})$ means that $\|\mathbf{e}_n\|_E2^{\theta n}\leq D$ for all $n\in\mathbb{N}$ and certain $D>0$. 

Hence
\begin{eqnarray*}
&\ & \|y_{N_0}\|_{E:K}  =  \|\{K(y_{N_0},2^n,A_0,A_1)\|_{E} 
\leq  C \|\{K(y_{N_0},2^n,A_0,A_1)\}\|_{\ell^r(\Phi^E)} \\
&=& C \left[ \sum_{k=0}^{\infty} \|e_{-k}\|_E^r K(y_{N_0}, \frac{1}{2^k},A_0,A_1)^r  + \sum_{k=1}^{\infty} \|e_{k}\|_E^r K(y_{N_0}, 2^k,A_0,A_1)^r \right]^{\frac{1}{r}} 
\\
&\leq & D C \left[ \sum_{k=0}^{\infty} 2^{k\theta r} K(y_{N_0}, \frac{1}{2^k},A_0,A_1)^r  + \sum_{k=1}^{\infty} 2^{-k\theta r} K(y_{N_0}, 2^k,A_0,A_1)^r \right]^{\frac{1}{q}} \\
 & = & D C \left[ \sum_{k=0}^{\infty} 2^{k\theta r} K(y_{N_0}, \frac{1}{2^k},A_0,A_1)^r  + 
 \sum_{k=1}^{\infty} 2^{(1-\theta)kr} K(y_{N_0}, \frac{1}{2^k},A_1,A_0)^r \right]^{\frac{1}{r}} \\
 & \leq & D C\left[ \sum_{k=0}^{N_0} 2^{k\theta r} 2^{-N_0r}+ \sum_{k=N_0+1}^{\infty} 2^{(\theta-1)kr} \right. \\
 &\ & \quad \quad \left. +
 \sum_{k=1}^{N_0} 2^{(1-\theta)kr} 2^{-N_0r}+ \sum_{k=N_0+1}^{\infty} 2^{(-\theta)kr} \right]^{\frac{1}{r}}\\
 &=&  D C \left[ \frac{2^{\theta r (N_0+1)}-1}{2^{\theta r}-1}2^{-N_0r}+ \frac{2^{(\theta-1)r(N_0+1)}}{1-2^{(\theta-1)r}} \right. \\
 &\ & \quad \quad \left. +
  \frac{2^{(1-\theta) r (N_0+1)}-1}{2^{(1-\theta) r}-1}2^{-N_0r}+ \frac{2^{-\theta r(N_0+1)}}{1-2^{-\theta r}}
 \right]^{\frac{1}{r}}\\
 &\leq&  D C \left[ 2^{(\theta-1)rN_0}\frac{2^{\theta r}}{2^{\theta r}-1} + \frac{2^{(\theta-1)r(N_0+1)}}{1-2^{(\theta-1)r}} 
\right. \\
 &\ & \quad \quad \left.  + 2^{-\theta rN_0}\frac{2^{(1-\theta)r}}{2^{(1-\theta) r}-1} + \frac{2^{-\theta r(N_0+1)}}{1-2^{-\theta r}}  \right]^{\frac{1}{r}},\\ 
\end{eqnarray*}
which converges to $0$ when $N_0$ goes to infinity. This demonstrates that the norms $\| \cdot\| _{A_0\cap A_1}$ and $\|\cdot \|_{E:K}$ are not equivalent on $A_0\cap A_1$, which implies that $A_0\cap A_1\neq (A_0,A_1)_{E:K}$.


\end{proof}

Note that condition $\{\min(1,2^n)\}_{n\in\mathbb{Z}}\in E$ is equivalent to a well known property of sequence spaces in the real-interpolation realm: Indeed $$\{\min(1,2^n)\}_{n\in\mathbb{Z}}\in E\Longleftrightarrow \ell^{\infty}(\max(1,2^{-n}))=\ell^{\infty}\cap \ell^{\infty}(2^{-n})\hookrightarrow E$$
which is equivalent to $E$ been $K$-nontrivial.


\begin{cor} \label{strict}
    Given a couple of quasi-Banach spaces $(A_0,A_1)$, we always have
    $$ A_0\cap A_1 \hookrightarrow (A_0,A_1)_{\theta,q} \hookrightarrow A_0+A_1.$$
    Moreover:
    \begin{itemize}
        \item[$(a)$] If $A_0\neq A_0+A_1$ or $A_1\neq A_0+A_1$, then $$  (A_0,A_1)_{\theta,q} \neq A_0+A_1.$$
    \item[$(b)$] If $A_0\cap A_1$ is not closed in $A_0+A_1$. then 
    $$(A_0,A_1)_{\theta,q} \neq A_0\cap A_1.$$     
    \end{itemize}
\end{cor}
The condition given in item 
$(b)$ of Corollary  \ref{strict}  is a natural condition since for any Banach couple $(A_0,A_1)$, if $A_0\cap A_1$ is closed in $A_0+A_1$ then $$(A_0, A_1)_{\theta, p}=A_0\cap A_1\quad  \mbox{for all parameters}.$$ 

The following result was proved in \cite[Theorem 14, item (a)]{AF} but the proof was not complete, as was pointed out  by F. Cobos (personal communication) to the second author of this paper. It is because of this that we include the statement and proof of the theorem here:

\begin{thm} \label{inter} Let $(A_0,A_1)$ be a couple of quasi-Banach spaces. Assume that $A_1\hookrightarrow A_0$, that  $A_1\neq A_0$ and that $A_1$ is not closed in $A_0$. Assume also that $0<\theta_0, \theta_1<1$, $\theta_0\neq \theta_1$, and $0<p<\infty$, $0<q\leq \infty$. Then  
 $$(A_0,A_1)_{\theta_0,p}\neq  (A_0,A_1)_{\theta_1,q}.$$ 
\end{thm}

\begin{proof}
This result follows from Corollary \ref{strict} and the fact that $A_0$ is a space of class $\mathcal{C}(0,(A_0,A_1))$, and $(A_0,A_1)_{\theta,p}$ is a space of class $\mathcal{C}(\theta,(A_0,A_1))$, which implies that we can use the following reiteration formula \cite[Theorem 3.11.5]{BL}:
\begin{equation}\label{reiteX}
(A_0,(A_0,A_1)_{\theta,p})_{\alpha,q}=(A_0,A_1)_{\alpha\theta,q}
\end{equation}
Indeed, if $0<\theta_1<\theta_0<1$, then $\theta_1=\alpha\theta_0$ for a certain $\alpha\in]0,1[$, which implies that 
\begin{equation}\label{reit2}
(A_0,A_1)_{\theta_1,q}=(A_0,A_1)_{\alpha\theta_0,q}=(A_0,(A_0,A_1)_{\theta_0,p})_{\alpha,q}
\end{equation}
Corollary \ref{strict} guarantees that $(A_0,A_1)_{\theta_0,p}$ is (a quasi-Banach space) strictly embedded into $A_0$. 
Moreover, for $p<\infty$, $(A_0,A_1)_{\theta_0,p}$ is not closed in $A_0$. 

Indeed, $p<\infty$ in conjunction with $A_1\hookrightarrow A_0$, imply that $(A_0,A_1)_{\theta,p}\subseteq \overline{A_1}^{A_0}$ and $A_1$ not closed in $A_0$ implies that $A_1$ is strictly contained into $X:=\overline{A_1}^{A_0}$. Here by $\overline{A_1}^{A_0}$ we mean the closure of $A_1$ in $A_0$. Moreover, $(A_0,A_1)_{\theta,p}=(X,A_1)_{\theta,p}$, where $X$ is doted with the norm of $A_0$. In particular, we can apply Corollary \ref{strict} to the couple $(X,A_1)$ to conclude that $$A_1\hookrightarrow (X,A_1)_{\theta,p}=(A_0,A_1)_{\theta,p} \hookrightarrow X$$
with strict inclusions. Now, $A_1$ is dense in $X$, which implies that  $(A_0,A_1)_{\theta,p}$ is not closed in $X$ and, henceforth, it is also not closed in $A_0$, as we wanted to demonstrate. 

 It follows that we can also apply Corollary \ref{strict} to $(A_0,(A_0,A_1)_{\theta_0,p})$, which, in conjuction with formula \eqref{reit2}, implies that $(A_0,A_1)_{\theta_1,q}$ is strictly embedded between $(A_0,A_1)_{\theta_0,p}$ and $A_0$. In particular, $$(A_0,A_1)_{\theta_1,q}\neq (A_0,A_1)_{\theta_0,p} .$$    
\end{proof}
\begin{rem}
    The conclusions of Theorem \ref{inter} also hold true if the roles of $A_0$ and $A_1$ are interchanged in the statement of the theorem, so that its hypotheses are substituted by: $A_0\hookrightarrow A_1$,  $A_1\neq A_0$ and $A_0$ is not closed in $A_1$. This is so because   
    $(A_0,A_1)_{\theta,q} = (A_1,A_0)_{1-\theta,q}$ for all $\theta\in ]0,1[$. 
\end{rem}

Theorem \ref{inter} has been stated only for ordered couples, so that generalizing it to arbitrary couples is a natural goal. An idea that partially works is to substitute the couple $(A_0,A_1)$ by the ordered couple $(A_0+A_1,A_0\cap A_1)$ and use the relations that exist between the interpolation spaces defined by this couple and the interpolation spaces associated to the original couple $(A_0,A_1)$.  Indeed, independently of our goals in this paper, in general, it is natural to ask how the interpolation couples of $(A_0,A_1)$, $(A_0+A_1, A_0)$, $(A_0, A_0\cap A_1)$ and $(A_0+A_1, A_0\cap A_1)$ are interrelated. A complete answer to this question was given in \cite{H} for Banach couples, and the problem has been studied in several other papers, always for the Banach setting \cite{AA,M,P}. 
Concretely, the following formulae are known for Banach couples $(A_0,A_1)$ and $(\theta,p)\in ([0,1]\times [1,\infty])\setminus (\{0,1\}\times [1,\infty))$:
\begin{equation}\label{sumintersection}
    (A_0+A_1,A_0\cap A_1)_{\theta,p}=\left\{\begin{array}{cccc}
   (A_0,A_1)_{\theta,p}+(A_0,A_1)_{1-\theta,p} & 0\leq \theta \leq 1/2 \\
     (A_0,A_1)_{\theta,p}\cap (A_0,A_1)_{1-\theta,p} & 1/2\leq \theta \leq 1 
\end{array}\right.
\end{equation}
Fortunately, this  formula also hold for quasi-Banach couples $(A_0,A_1)$ and $(\theta,p)\in ]0,1[\times ]0,\infty[$, as we demonstrate in the Appendix of this paper (see Theorem \ref{sumainterseccion} below). For the proof of the following theorem, we also use the reiteration formulae
\begin{equation}\label{reit3}
((A_0,A_1)_{\theta,p},(A_0,A_1)_{\theta,q})_{\alpha,r}=(A_0,A_1)_{\theta, r},\text{ where } \frac{1}{r} = 
\frac{1-\alpha}{p}+\frac{\alpha}{q},
\end{equation}
and
\begin{equation}\label{reit4}
((A_0,A_1)_{\theta_0,p_0},(A_0,A_1)_{\theta_1,p_1})_{\alpha,r}=(A_0,A_1)_{(1-\alpha)\theta_0+\alpha\theta_1, r}.
\end{equation}
\begin{thm} \label{internonordered} Let $(A_0,A_1)$ be a couple of quasi-Banach spaces. Assume that $A_1\cap A_0$ is not closed in $A_0+A_1$. Then, if $\theta\neq 1/2$, $0<\theta<1$, and $0<p<\infty$, we have that:
$$(A_0,A_1)_{\theta,p}\neq  (A_0,A_1)_{1-\theta,p}=(A_1,A_0)_{\theta,p}.$$ 
Moreover, if $0<\theta_0, \theta_1<1$, $\theta_0\neq \theta_1$, and $0<p,q<\infty$, then either 
$$ (A_0,A_1)_{\theta_0,p}\neq  (A_0,A_1)_{\theta_1,q}$$
or 
$$ (A_1,A_0)_{\theta_0,p}\neq  (A_1,A_0)_{\theta_1,q}.$$

\end{thm}
\begin{proof}
By hypothesis, $A_0\cap A_1$ is not closed in $A_0+A_1$, so that we can use Theorem \ref{inter} to conclude that, for $\theta\neq 1/2$, $$(A_0+A_1,A_0\cap A_1)_{\theta,p}\neq (A_0+A_1,A_0\cap A_1)_{1-\theta,p}$$
Hence, taking into account \eqref{sumintersection}, we have that 
$$
(A_0,A_1)_{\theta,p}+(A_0,A_1)_{1-\theta,p} \neq 
     (A_0,A_1)_{\theta,p}\cap (A_0,A_1)_{1-\theta,p} 
$$
which implies that $$(A_0,A_1)_{\theta,p}\neq (A_0,A_1)_{1-\theta,p}.$$
This proves the first claim. 
 
Using again Theorem \ref{inter} we conclude that, for $\theta_0\neq \theta_1$, $0<p,q<\infty$,  we have that $$(A_0+A_1,A_0\cap A_1)_{\theta_0,p}\neq (A_0+A_1,A_0\cap A_1)_{\theta_1,q}.$$
Hence, taking into account \eqref{sumintersection}, we can consider the following cases:

\noindent \textbf{Case 1: } $0<\theta_0,\theta_1\leq 1/2$.  

Then \[
    (A_0,A_1)_{\theta_0,p}+(A_0,A_1)_{1-\theta_0,p} \neq (A_0,A_1)_{\theta_1,q}+(A_0,A_1)_{1-\theta_1,q}
    \]
    Thus, if $(A_0,A_1)_{\theta_0,p}=(A_0,A_1)_{\theta_1,q}$, then $$(A_0,A_1)_{1-\theta_0,p}\neq(A_0,A_1)_{1-\theta_1,q}.$$ In other words, $$(A_1,A_0)_{\theta_0,p}\neq(A_1,A_0)_{\theta_1,q}.$$
    
    \noindent \textbf{Case 2: }  $1/2<\theta_0,\theta_1< 1$.  
    
    Then 
    \[
    (A_0,A_1)_{\theta_0,p}\cap (A_0,A_1)_{1-\theta_0,p} \neq (A_0,A_1)_{\theta_1,q}\cap (A_0,A_1)_{1-\theta_1,q}
    \]
    so that, again, if $(A_0,A_1)_{\theta_0,p}=(A_0,A_1)_{\theta_1,q}$, then $$(A_0,A_1)_{1-\theta_0,p}\neq(A_0,A_1)_{1-\theta_1,q}.$$ In other words, $$(A_1,A_0)_{\theta_0,p}\neq(A_1,A_0)_{\theta_1,q}.$$

\noindent \textbf{Case 3: }  $\theta_0<1/2\leq \theta_1$. 

If $(A_0,A_1)_{\theta_0,p}=(A_0,A_1)_{\theta_1,q}$, then for each $0<\alpha<1$ and $0<r<\infty$ we have that 
\[
((A_0,A_1)_{\theta_0,p},(A_0,A_1)_{\theta_0,q})_{\alpha,r}= ((A_0,A_1)_{\theta_1,q},(A_0,A_1)_{\theta_0,q})_{\alpha,r}.
\]
Hence, using \eqref{reit3} and  \eqref{reit4}, we get
\[
(A_0,A_1)_{\theta_0,s}=(A_0,A_1)_{(1-\alpha)\theta_1+\alpha\theta_0,r},
\]
where $\frac{1}{s}=\frac{1-\alpha}{p}+\frac{\alpha}{q}$. 
Now, taking $\alpha$ such that $\theta_0<(1-\alpha)\theta_1+\alpha\theta_0<1/2$, we can apply the result we got for case 1 to claim that
\[
(A_1,A_0)_{\theta_0,s}\neq (A_1,A_0)_{(1-\alpha)\theta_1+\alpha\theta_0,r},
\]
which implies that 
\[
((A_1,A_0)_{\theta_0,p},(A_1,A_0)_{\theta_0,q})_{\alpha,r}\neq  ((A_1,A_0)_{\theta_1,q},(A_1,A_0)_{\theta_0,q})_{\alpha,r}.
\]
Hence 
\[
(A_1,A_0)_{\theta_0,p}\neq (A_1,A_0)_{\theta_1,q}.
\]
\end{proof}

\begin{thm}
Let $(A_0,A_1)$ be a couple of quasi-Banach spaces. Assume that $A_1\cap A_0$ is not closed in $A_0+A_1$. 
If $0<\theta_0<1/2<\theta_1<1$, and $0<p,q<\infty$, then 
 $$ (A_0,A_1)_{\theta_0,p}\neq   (A_0,A_1)_{\theta_1,q}.$$
Moreover, if $$ (A_0,A_1)_{\theta_0,p}=   (A_0,A_1)_{\theta_1,q}$$
with $0<\theta_0<\theta_1\leq \frac{1}{2}$ or $\frac{1}{2}\leq \theta_0<\theta_1<1$, then all the spaces $(A_0,A_1)_{\theta,r}$, $\theta\in [\theta_0,\theta_1]$, $0<r<\infty$ collapse into a unique space, and  
$$ (A_1,A_0)_{\phi_0,r_0}\neq   (A_1,A_0)_{\phi_1,r_1}.$$
whenever $\phi_0,\phi_1\in[\theta_0,\theta_1]$, $\phi_0\neq \phi_1$, and $0<r_0,r_1<\infty$. 
\end{thm}
\begin{proof}
    This result follows from Theorem \ref{internonordered} and the reiteration formula \eqref{reit4}. Indeed, if $$(A_0,A_1)_{\theta_0,p}=   (A_0,A_1)_{\theta_1,q}=E$$ then \eqref{reit4} implies that all the spaces $(A_0,A_1)_{\theta,r}$, $\theta\in [\theta_0,\theta_1]$, $0<r<\infty$ collapse into a unique space. This is so because, if $\theta=(1-\alpha)\theta_0+\alpha\theta_1$ with $0< \alpha< 1$, then 
    \begin{eqnarray*}
    E &=& (E,E)_{\alpha,r}=((A_0,A_1)_{\theta_0,p},(A_0,A_1)_{\theta_1,q})_{\alpha,r}\\
    &=& (A_0,A_1)_{(1-\alpha)\theta_0+\alpha\theta_1, r}=(A_0,A_1)_{\theta, r}.
    \end{eqnarray*}
But this is impossible if $\theta_0<\frac{1}{2}<\theta_1$ since in such case there is a value $\theta <\frac{1}{2}$ such that $[\theta,1-\theta]\subset ]\theta_0,\theta_1[$ and Theorem \ref{internonordered} implies that $$ (A_0,A_1)_{\theta,r}\neq    (A_0,A_1)_{1-\theta,r}.$$ 
    
On the other hand, if $$ (A_0,A_1)_{\theta_0,p}=   (A_0,A_1)_{\theta_1,q}$$
with $0<\theta_0<\theta_1\leq \frac{1}{2}$ or $\frac{1}{2}\leq \theta_0<\theta_1<1$, and $0<p,q<\infty$,  then for each $\phi_0,\phi_1 \in [\theta_0,\theta_1]$  and for each $0<r_0,r_1<\infty$ we have that 
$$ (A_0,A_1)_{\phi_0,r_0}=   (A_0,A_1)_{\phi_1,r_1}$$
and Theorem \ref{internonordered} implies that, if $\phi_0\neq \phi_1$, then 
$$ (A_1,A_0)_{\phi_0,r_0}\neq   (A_1,A_0)_{\phi_1,r_1}.$$
\end{proof}

\begin{thm}\label{interpol} Let $(A_0,A_1)$ be a couple of quasi-Banach spaces, and  assume that  
$A_0\cap A_1$ is not closed in $A_0+A_1$.  
Let $\theta\in ]0,1[$ and assume that $0<p,q\leq \infty$ are such that $(A_0,A_1)_{\theta,p}\neq  (A_0,A_1)_{\theta,q}$, and that $r_1,r_2\in [q,p]$, $r_1\neq r_2$. Then $$(A_0,A_1)_{\theta,r_1}\neq (A_0,A_1)_{\theta,r_2}.$$ 
\end{thm}
\begin{proof}
 
Let us assume $p>q>0$. We use the reiteration formula \eqref{reit4} 
 which, in conjunction with Corollary \ref{strict}, implies that, if $(A_0,A_1)_{\theta,p}\neq  (A_0,A_1)_{\theta,q}$  then all interpolation spaces $(A_0,A_1)_{\theta, r}$ with $r\in ]q,p[$ are strictly embedded between the spaces $(A_0,A_1)_{\theta,p}$ and  $(A_0,A_1)_{\theta,q}$ (note that, if $p>q$, then $(A_0,A_1)_{\theta,q}\hookrightarrow (A_0,A_1)_{\theta,p}$ since $\ell_q\hookrightarrow \ell_p$). This is so because   $\frac{1}{r}=\frac{1-\eta}{p}+\frac{\eta}{q}$ with $0\leq\eta\leq1$ is another way to claim that $\frac{1}{r} \in [\frac{1}{p},\frac{1}{q}]$ which is equivalent to $r\in [q,p]$. Thus, if $r_1,r_2\in ]q,p[$, $r_1<r_2$, then $(A_0,A_1)_{\theta, r_1}$ is strictly embedded in $(A_0,A_1)_{\theta,p}$ and  henceforth $(A_0,A_1)_{\theta,r_2}$ is strictly embedded between the spaces $(A_0,A_1)_{\theta,p}$ and  $(A_0,A_1)_{\theta,r_1}$, since $r_1<r_2<p$. In particular,  $(A_0,A_1)_{\theta, r_1}\neq  (A_0,A_1)_{\theta, r_2}$.
\end{proof}
\begin{cor}\label{corfin}
    Let $(A_0,A_1)$ be a couple of quasi-Banach spaces, and  assume that  
$A_0\cap A_1$ is not closed in $A_0+A_1$, and let $\theta\in ]0,1[$. Then the spaces $(A_0,A_1)_{\theta,p}$, $0<p\leq \infty$ either collapse (all of them are equal) or they are pairwise distinct. 
\end{cor}
\begin{proof}
    Assume that the spaces $(A_0,A_1)_{\theta,p}$ do not collapse. Then there are $0<q<p\leq \infty$ such that  $(A_0,A_1)_{\theta,p}\neq (A_0,A_1)_{\theta,q}$.  Take $t,s$ such that  $0<s<q<p\leq t\leq \infty$.  Then $$(A_0,A_1)_{\theta,s}\neq (A_0,A_1)_{\theta,t}$$ since 
    $$(A_0,A_1)_{\theta,s} \hookrightarrow (A_0,A_1)_{\theta,q} \hookrightarrow (A_0,A_1)_{\theta,p}\hookrightarrow (A_0,A_1)_{\theta,t}$$
    Thus, Theorem \ref{interpol} implies that the spaces $(A_0,A_1)_{\theta,r}$ are pairwise distinct for $r\in [s,t]$. 
\end{proof}
\section{Applications}
An interesting application of the results above is the proof of strict inclusions for the embeddings of the operator ideals defined by approximation numbers $a_n(T)$, i.e., consider  the space of all bounded operators $T:X \to Y$ where the sequence $(a_n(T))$ belongs to the Lorentz sequence space $\ell_{p,q}$. This operator ideal is denoted by $\mathcal{L}_{p,q}^{(a)}(X,Y)$  and is defined as: 

\[
\mathcal{L}_{p,q}^{(a)}(X,Y)=\{T\in \mathcal{L}(X,Y): \|T\|_{p,q}=\left(\sum_{n=1}^{\infty}(n^{\frac{1}{p}-\frac{1}{q}}a_n(T))^q\right)^{1/q}<\infty\}.
\]
Here, $X,Y$ are Banach spaces and $a_n(T)=\inf_{\text{\rm{rank}} (R)<n}\|T-R\|$ is the $n$-th approximation number of the operator $T:X\to Y$. 

In particular, it is known (see \cite[Theorem 1]{K}) that 
\begin{equation}\label{opersnumb}
    \mathcal{L}_{p,q}^{(a)}(X,Y)= (\mathcal{L}_{p_0,p_0}^{(a)}(X,Y),\mathcal{L}_{p_1,p_1}^{(a)}(X,Y))_{\theta,q},\quad \frac{1}{p}=\frac{1-\theta}{p_0}+\frac{\theta}{p_1}
\end{equation}
and that (see \cite[Lemma 4.2]{CK}), for $0<p_0<p_1<\infty$ and $0<\theta<1$, 
\begin{equation}
\mathcal{L}_{p_0,p_0}^{(a)}(X,Y)\hookrightarrow\mathcal{L}_{p_1,p_1}^{(a)}(X,Y)
\end{equation} 
with strict inclusion, and $\mathcal{L}_{p_0,p_0}^{(a)}(X,Y)$ is not closed in $\mathcal{L}_{p_1,p_1}^{(a)}(X,Y)$ so that we can use Theorems \ref{strict}, \ref{inter}, \ref{interpol}, as well as Corollary \ref{corfin} with the couple $ (\mathcal{L}_{p_0,p_0}^{(a)}(X,Y),\mathcal{L}_{p_1,p_1}^{(a)}(X,Y))_{\theta,q}.$  Concretely, the following holds:

\begin{thm}\label{operadores}
    If $p\neq p^*$, $0<p,p^*<\infty$, then $$\mathcal{L}_{p,q_1}^{(a)}(X,Y)\neq \mathcal{L}_{p^*,q_2}^{(a)}(X,Y)$$
    for all $0<q_1,q_2<\infty$.
    Furthermore, for any given $p>0$, 
    the spaces $$\mathcal{L}_{p,r}^{(a)}(X,Y), \quad 0<r<\infty $$ either collapse (are all equal) or they are pairwise distinct. 
\end{thm}

    In his late years, A. Pietsch posed as an open problem to prove that $$\mathcal{L}_{p,q}^{(a)}(X,Y)\neq \mathcal{L}_{p^*,q^*}^{(a)}(X,Y)$$ whenever $(p,q)\neq (p^*,q^*)$, for arbitrary Banach spaces $X,Y$. The problem was solved, in positive sense, in \cite{CK}. To prove it, the authors solved Lions problem for general quasi-Banach couples. Their result was quite technical, when compared to the results we have presented in this paper, and required the extra hypothesis that the couple $(A_0,A_1)$ is a Gagliardo couple, which is not assumed here. Of course, as a compensation to the technicality of their proof, they solved the problem for all cases while, with our elementary approach, we only obtain a partial solution to Lions problem, but we still solve all cases except one of Pietsch's problem about the ideals of operators $\mathcal{L}_{p,q}^{(a)}(X,Y)$. Note that a full solution to this problem also follows directly as a consequence of Oikhberg's results in \cite{O}, but these are also quite technical.

Let us now focus our attention on operator ideals defined with general $s$-numbers sequences. Recall that a rule that associates to every operator $T$ a sequence $(s_n(T))$ of non-negative numbers is called a $s$-numbers sequence (or $s$-numbers scale) if the following holds:
\begin{itemize}
    \item $\|T\|=s_1(T)\geq s_2(T)\geq \cdots \geq 0$.
    \item $s_{n+m-1}(T+S)\leq s_n(T)+s_m(S)$ for $S,T\in \mathcal{L}(X,Y)$, $n,m=1,2,\cdots$.
    \item $s_{n+m-1}(TS)\leq s_n(T)s_m(S)$ for $S\in \mathcal{L}(X,Y)$, $T\in \mathcal{L}(Y,Z)$, $n,m=1,2,\cdots$.
    \item $s_n(1_d:\ell_2^n\to \ell_2^n)=1$, $n=1,2,\cdots$.
    \item $s_n(T)=0$ if $\text{\rm{rank}} (T)<n$. 
\end{itemize}
Associated to every $s$-numbers sequence, we can consider the spaces
\[
\mathcal{L}_{p,q}^{(s)}(X,Y)=\{T\in \mathcal{L}(X,Y): \|T\|_{p,q}=\left(\sum_{n=1}^{\infty}(n^{\frac{1}{p}-\frac{1}{q}}s_n(T))^q\right)^{1/q}<\infty\}.
\]
and we can ask under which conditions these spaces are pairwise distinct. In the case of approximation numbers, we have already (partially) solved the problem, with an elementary proof, by using that these spaces can be identified as interpolation spaces. Unfortunately, for arbitrary $s$-numbers, formula \eqref{opersnumb}   does not hold in general and we can't proceed in the same way.

Indeed, K\"{o}nig has studied the relationship that exists between the operator ideals and interpolation.  Concretely, in \cite[Theorem 1]{K} he proved that, for any $s$-numbers sequence,  
    \begin{equation}\label{opersnumb2}    (\mathcal{L}_{p_0,p_0}^{(s)}(X,Y),\mathcal{L}_{p_1,p_1}^{(s)}(X,Y))_{\theta,q}\subseteq \mathcal{L}_{p,q}^{(s)}(X,Y),\quad \frac{1}{p}=\frac{1-\theta}{p_0}+\frac{\theta}{p_1}
\end{equation}
and he also proved that there are $s$-numbers sequences such that equality does not hold in \eqref{opersnumb2}. In particular, equality does not hold for the operator ideal defined with Gelfand numbers, $\mathcal{L}_{p,q}^{(c)}(X,Y)$. Here for an operator $T\in \mathcal{L}(X,Y)$, the $n$-th  \textit{Gelfand numbers} are defined by  $$ c_n(T):=a_n(J_YT)$$ where $J_Y$ is the canonical injection from  $Y$ into $Y^{inj}:=\ell^{\infty}(B_{Y^*})$, where $Y^*$ denotes the dual space of $Y$ and $B_{Y^*}=\{f\in Y^*:\|f\|\leq 1\}$ is its unit ball.
 In the following the main idea of the proof of this claim is given.\\
 In \cite[Lemma 14.2.14]{POI} is shown that, if $(\sigma_n)$ is a non-increasing sequence of non-negative numbers and $(\sigma_n)\in \ell^p(\mathbb{N})$ for a certain $0<p<\infty$, then $$\lim_{n\to\infty}n^{1/p}\sigma_n=0.$$ Thus if $0<p<\infty$ and $T\in \mathcal{L}_{p,p}^{(s)}(X,Y)$, then $\lim_{n\to\infty}n^{1/p}s_n(T)=0$. 
If we apply this to the sequence of Gelfand numbers $c_n(T)$ and we take into account that $a_n(T)\leq (2n)^{\frac{1}{2}}c_n(T)$ holds for $T\in \mathcal{L}(X,Y)$ for arbitrary Banach spaces $X,Y$, we get that $T\in \mathcal{L}_{2,2}^{(c)}(X,Y)$ implies $0\leq \lim_{n\to \infty}a_n(T)\leq \lim_{n\to \infty}(2n)^{1/2}c_n(T)=0$ so that  
\[
\mathcal{L}_{2,2}^{(c)}(X,Y)\subseteq \mathcal{L}^{\mathcal{A}}(X,Y).
\]
Here $\mathcal{L}^{\mathcal{A}}(X,Y)$ denotes the approximable operators $T:X\to Y$  (i.e., $\mathcal{L}^{\mathcal{A}}(X,Y)$ is the closure in $\mathcal{L}(X,Y)$ of
$\mathcal{F}(X,Y)=\{T\in \mathcal{L}(X,Y): \text{\rm{rank}}(T)<\infty\}$, which is complete with respect to the norm of $\mathcal{L}(X,Y)$.)
On the other hand, in \cite[Remark after Lemma 14.2.14]{POI} is stated that, as a consequence of Enflo's counterexample to the approximation property on Banach spaces, there exists a Banach space $E_0$ (the so-called Enflo's space) such that, if $p>2$, then 
\begin{equation} \label{Enf}
  \mathcal{L}_{p,p}^{(c)}(E_0)\not\subseteq \mathcal{L}^{\mathcal{A}}(E_0)  
\end{equation}
(see \cite{Ku} for a proof of this claim). Thus, if we could put an equality in formula \eqref{opersnumb2}, we would have that, for $2<p,q<\infty$ such that $\frac{1}{p}=\frac{1-\theta}{2}+\frac{\theta}{q}$, $0<\theta<1$, 
\begin{equation}\label{Asuman}
\mathcal{L}_{p,p}^{(c)}(E_0)=(\mathcal{L}_{2,2}^{(c)}(E_0),\mathcal{L}_{q,q}^{(c)}(E_0))_{\theta,p}\subseteq \mathcal{L}^{\mathcal{A}}(E_0) 
\end{equation}
which contradicts \eqref{Enf}. Note that last inclusion in  \eqref{Asuman} follows from the fact that 
$\mathcal{L}_{2}^{(c)}(E_0)$ is dense in $(\mathcal{L}_{2}^{(c)}(E_0),\mathcal{L}_{q}^{(c)}(E_0))_{\theta,p}$, so that if we take $T\in (\mathcal{L}_{2}^{(c)}(E_0),\mathcal{L}_{q}^{(c)}(E_0))_{\theta,p}$, we can approximate this operator in the quasi-norm of the interpolation space as well as in the norm of $\mathcal{L}_{2}^{(c)}(E_0)$, which is embedded into  $\mathcal{L}^{\mathcal{A}}(E_0)$ so that $T$ is approximable. This proves the claim. 

In his book \cite{Kbook}, K\"{o}nig also studied the relationship between interpolation and other types of operator ideals, and in all cases except for $\mathcal{L}_{p,q}^{(a)}$, the general situation is that only one inclusion holds. 
For example, it is known that \cite[p. 123]{Kbook}:
\[
N_p(X,Y)\subseteq (N_{p_0}(X,Y),N_{p_1}(X,Y))_{\theta,p}\,\,
\]
for $0<p_0<p_1\leq 1$ $0<\theta<1$, $\frac{1}{p}=\frac{1-\theta}{p_0}+\frac{\theta}{p_1}$, 
and \cite[p. 126]{Kbook}:
\[
(\Pi_{p_0,r}(X,Y),\Pi_{p_1,r}(X,Y))_{\theta,p}\subseteq \Pi_{p,r}(X,Y)
\]
for $1\leq r<p_0<p_1<\infty$ $0<\theta<1$, $\frac{1}{p}=\frac{1-\theta}{p_0}+\frac{\theta}{p_1}$,
where by $N_p(X,Y)$ we mean $p$-nuclear operators, and $\Pi_{p,q}(X,Y)$ denotes $(p,q)$-summing operators.

The relationships between these different operator ideals are studied in full detail in \cite{POI} and \cite{K}. In the following we  mention a strict inclusion between operator ideal generated by Gelfand numbers and $(p,q)$-summing maps.  It is known that in case of Hilbert space we have the following inclusion $$ \mathcal{L}_{2p/q,p}^{(a)} (\ell_2,\ell_2) \subseteq \Pi_{p,q}(\ell_2,\ell_2)$$ for $2< q<p< \infty$ (see \cite{B}). This result was generalized for the operators acting between Banach spaces, in particular it was proven in \cite{F} that
$$\mathcal{L}_{r,p}^{(c)} (X,Y) \subseteq \Pi_{p,q}(X,Y)$$ 
for $1\leq q<p<\infty$ and $1/r:=1/p\, \max(q/2, 1)$. Moreover in the same paper it was mentioned that these inclusions are strict.  This follows from the fact that (see \cite{POI}, section 6.5.4, page 99) the embedding map $I: \ell_1 \to \ell_2$ is absolutely $(1,1)$-summing. Therefore $$I\in \Pi_{p,q}(\ell_1, \ell_2)\quad\mbox{for} \quad 1 \leq q\leq p < \infty.$$ On the other hand $I$ is not compact, so that $$I \notin \mathcal{L}_{r,w}^{(c)} (\ell_1,\ell_2)\quad \mbox{for}\quad 0<r, \,\, w< \infty.$$

In his book on the history of Banach spaces and linear operators, when dealing with interpolation of operator ideals, Pietsch commented (see \cite[p. 436]{Phbs}):
``Apart from $\mathcal{L}_{p,p}^{(a)}$, I do not know any other $1$-parameter scale of ideals $\frak{U}_p$ that could be reproduced by real or complex interpolation: $(\frak{U}_{p_0},\frak{U}_{p_0})_{\theta, p}= \frak{U}_{p}$ or $[\frak{U}_{p_0},\frak{U}_{p_0}]_{\theta}= \frak{U}_{p}$''
And  this is still the situation almost two decades later. Thus, the approach of using interpolation spaces to face the analogous of Lions problem with the operator ideals $\mathcal{L}_{p,q}^{(s)}$ seems to be adequate only for  $\mathcal{L}_{p,q}^{(a)}$.

On the other hand, using stronger tools, in some cases we know that the spaces $\mathcal{L}_{p,q}^{(s)}(X,Y)$ are pairwise distinct. For example, it follows from \cite[Theorem 1.6]{O} that, if the Banach spaces $X$ and $Y$ have no non-trivial cotype and no non-trivial type, res\-pectively, then for every non-increasing sequence $(\varepsilon_n)\in c_0(\mathbb{N})$ there is $T\in\mathcal{L}(X,Y)$ such that 
$$\frac{1}{50}\varepsilon_n\leq c_n(T)\leq 4\varepsilon_{[4n/5]},$$ and this implies that the spaces $\mathcal{L}_{p,q}^{(c)}(X,Y)$ are pairwise distinct. Indeed, if we take 
$\varepsilon_n=n^{-1/p}(1+\log n)^{-1/q}$ we have that $(n^{\frac{1}{p}-\frac{1}{q}}\varepsilon_n)^q=n^{-1}(1+\log n)^{-1}\not\in \ell^1(\mathbb{N})$ since $\int_2^{\infty}\frac{dx}{x\log x}=\int_{\log 2}^{\infty}\frac{du}{u}=\infty$, 
and, for $q^*>q$, 
$(n^{\frac{1}{p}-\frac{1}{q^*}}\varepsilon_n)^{q^*}=n^{-1}(1+\log n)^{-q^*/q}\in \ell^1(\mathbb{N})$ since 
 $\int_2^{\infty}\frac{dx}{x(\log x)^{q^*/q}}=\int_{\log 2}^{\infty}\frac{du}{u^{q^*/q}}<\infty$.  Now, $(\varepsilon_n)\asymp (\varepsilon_{[4n/5]})$ and henceforth, if $T:X\to Y$ satisfies $\frac{1}{50}\varepsilon_n\leq c_n(T)\leq 4\varepsilon_{[4n/5]}$ then $T\in \mathcal{L}_{p,q^*}^{(c)}(X,Y)$ for all $q^*>q$ and $T\not \in \mathcal{L}_{p,q}^{(c)}(X,Y)$. Moreover, if $p^*>p$, 
 $(n^{\frac{1}{p^*}-\frac{1}{q^*}}\varepsilon_n)^{q^*}=n^{-(q^*(\frac{1}{p}-\frac{1}{p^*})+1)}(1+\log n)^{-q^*/q}\in \ell^1(\mathbb{N})$ and 
 we also have that $T \in \mathcal{L}_{p^*,q^*}^{(c)}(X,Y)$ for every $q^*$. Hence the spaces $\mathcal{L}_{p,q}^{(c)}(X,Y)$ are pairwise distinct.

If we renounce to solve Lions problem for $\mathcal{L}_{p,q}^{(s)}(X,Y)$ in its full generality, we can at least face a simpler question using BCT. Concretely, we can study under which conditions $\mathcal{L}_{p,q}^{(s)}(X,Y) \hookrightarrow \mathcal{L}(X,Y)$ with strict inclusion, a property that  holds true if we can guarantee that the $s$-numbers sequence $s_n(T)$ decays slowly to zero. It is with this objective in mind that we prove the following lethargy property of s-numbers.

\begin{lem}\label{sucesiones} Given $\{\varepsilon_n\}\searrow   0$ and $\{h(n)\}$ an increasing sequence of natural numbers satisfying  $n\leq h(n)$ for all $n$, there exists a sequence
$\{\xi_n\}\searrow   0$ such that $\varepsilon_n\leq \xi_n$ and $\xi_n\leq 2\xi_{h(n)}$ for all $n$.
\end{lem}

\begin{proof} See \cite[Lemma 2.3]{AO}.

\end{proof}

\begin{thm}\label{let} 
Given $X,Y$ Banach spaces, the following are equivalent claims:
\begin{itemize}
\item[$(a)$]  $s_n(S(\mathcal{L}^{\mathcal{A}}(X,Y)))=\sup_{\|T\|=1}s_n(T)>c$ for all  $n\in\mathbb{N}$ and a certain constant $c>0$.
\item[$(b)$] For every non-increasing sequence $\{\varepsilon_n\}\in c_0$, there are approximable operators $T\in \mathcal{L}^{\mathcal{A}}(X,Y)$ such that 
\[
s_n(T)\neq \mathbf{O}(\varepsilon_n).
\]
\end{itemize}
\end{thm}
\begin{proof} $(a)\Rightarrow (b)$.  Let us assume, on the contrary, that $s_n(T)= \mathbf{O}(\varepsilon_n)$ for all $T\in \mathcal{L}^{\mathcal{A}}(X,Y)$ and certain sequence $\{\varepsilon_n\}\in c_0$. Thanks to Lemma \ref{sucesiones} we may assume, with no loss of generality, that $\varepsilon_n\leq 2\varepsilon_{2n+1}$ for all $n$. 

Clearly, $s_n(T)= \mathbf{O}(\varepsilon_n)$ can be reformulated as $$\mathcal{L}^{\mathcal{A}}(X,Y)=\bigcup_{m=1}^{\infty} \Gamma_m,$$ where $$\Gamma_m=\{T\in \mathcal{L}^{\mathcal{A}}(X,Y): s_n(T)\leq m\varepsilon_n \text{ for all }n\in\mathbb{N}\}$$
Now, $\Gamma_m$ is a closed subset of $\mathcal{L}^{\mathcal{A}}(X,Y)$ for all $m$ and the Baire category theorem implies that $\Gamma_{m_0}$ has nonempty interior for some $m_0\in\mathbb{N}$. On the other hand, $\Gamma_m=-\Gamma_m$ since $s_n(T)=s_n(-T)$ for all $T\in \mathcal{L}^{\mathcal{A}}(X,Y)$.  
Furthermore,
\[
\textbf{conv}(\Gamma_m)\subseteq \Gamma_{2m},
\]
since, if $T,S\in \Gamma_m$ and $\lambda\in [0,1]$, then 
\begin{eqnarray*}
s_{2n-1}(\lambda T +(1-\lambda)S) &\leq&  s_n(\lambda T)+ s_n( (1-\lambda)S)\\
&\leq &   \lambda s_n(T)+ (1-\lambda) s_n(S) \\
&\leq&  \lambda  m\varepsilon_n+(1-\lambda) m\varepsilon_n \\
&=&  m\varepsilon_n 
\end{eqnarray*} 
Now, given $j\in\mathbb{N}$, there is $n$ such that $2n-1\leq j<2n+1$, so that 
\[
s_{j}(\lambda T +(1-\lambda)S)\leq s_{2n-1}(\lambda T +(1-\lambda)S)\leq m\varepsilon_n \leq 2m\varepsilon_{2n+1} \leq 2m\varepsilon_j.
\]
Thus, if $B_{\mathcal{L}^{\mathcal{A}}(X,Y)}(T_0,r)=\{T\in \mathcal{L}^{\mathcal{A}}(X,Y):\|T_0-T\|<r\}\subseteq \Gamma_{m_{0}}$, then $\textbf{conv}(B_{\mathcal{L}^{\mathcal{A}}(X,Y)}(T_0,r)\cup B_{\mathcal{L}^{\mathcal{A}}(X,Y)}(-T_0,r))\subseteq \Gamma_{2m_0}$. In particular, 
\[
\frac{1}{2}(B_{\mathcal{L}^{\mathcal{A}}(X,Y)}(T_0,r)+B_{\mathcal{L}^{\mathcal{A}}(X,Y)}(-T_0,r))\subseteq \Gamma_{2m_0}.
\]
Now, it is clear that 
\[
B_{\mathcal{L}^{\mathcal{A}}(X,Y)}(0,r)\subseteq \frac{1}{2}(B_{\mathcal{L}^{\mathcal{A}}(X,Y)}(T_0,r)+B_{\mathcal{L}^{\mathcal{A}}(X,Y)}(-T_0,r)),
\]
since, if $T\in B_{\mathcal{L}^{\mathcal{A}}(X,Y)}(0,r)$, then $\|-T_0+ (T+T_0)\|=\|T_0+(T-T_0)\|=\|T\|\leq r$, so that  $T+T_0\in B_{\mathcal{L}^{\mathcal{A}}(X,Y)}(T_0,r)$, $T-T_0\in B_{\mathcal{L}^{\mathcal{A}}(X,Y)}(-T_0,r)$, and $T=\frac{1}{2}((T-T_0)+(T+T_0))$. Thus, 
\[
B_{\mathcal{L}^{\mathcal{A}}(X,Y)}(0,r)\subseteq  \Gamma_{2m_0}.
\]
This means that, if $T\in \mathcal{L}^{\mathcal{A}}(X,Y)\setminus \{0\}$, then 
\[
s_n(T)\leq 2m_0\varepsilon_n \text{ for all } n=1,2,\cdots .
\]
Hence 
\[
s_n(T)= s_n(\frac{\|T\|}{r}\frac{rT}{\|T\|}) \leq \frac{\|T\|}{r}2m_0\varepsilon_n \text{ for all } n=1,2,\cdots .
\]
On the other hand, $(a)$ implies that, for each $n$, there is $T_n\in S(\mathcal{L}^{\mathcal{A}}(X,Y))$ such that $$s_n(T_n)>c,$$ so that 
\[
c< s_n(T_n)\leq  \frac{1}{r}2m_0\varepsilon_n \text{ for all } n=1,2,\cdots,
\] 
which is impossible, since $\varepsilon_n$ converges to $0$ and $c>0$. This proves $(a)\Rightarrow (b)$. 

Let us demonstrate the other implication. Assume that $(a)$ does not hold. Then $$\{c_n=s_n(S(\mathcal{L}^{\mathcal{A}}(X,Y)))\}\in c_0.$$
In particular, if $T\in \mathcal{L}^{\mathcal{A}}(X,Y)$, $T\neq 0$, then
\[
s_n(\frac{T}{\|T\|})\leq s_n(S(\mathcal{L}^{\mathcal{A}}(X,Y)))=c_n \text{ for all } n\in\mathbb{N},
\]
so that 
\[
s_n(T) \leq \|T\| s_n(\frac{T}{\|T\|})\leq \|T\| c_n \text{ for all } n\in\mathbb{N},
\]
and $s_n(T)=\mathbf{O}(c_n)$ for all $x\in X$. This proves $(b)\Rightarrow (a)$. 
\end{proof}

As is well-known (and easy to prove, see \cite[p. 83]{Peigem}) the approximation numbers $a_n(T)$ satisfy $a_n(T)\geq s_n(T)$ for all $n$ and any $s$-numbers sequence. Hence they are the most natural $s$-numbers sequence to satisfy a lethargy result like Theorem \ref{let}. Indeed, in \cite{AO}, Corollary 6.23 and Theorem 3.4, is proved that, for arbitrary (infinite dimensional) Banach spaces $X,Y$ and for any decreasing sequence $(\varepsilon_n)\in c_0$ there are approximable operators $T\in \mathcal{L}(X,Y)$ satisfying $a_n(T)\geq \varepsilon_n$ for all $n=1,2,\cdots$.  In particular, this implies that $\mathcal{L}_{p,q}^{(a)}(X,Y)\hookrightarrow\mathcal{L}(X,Y)$ with strict inclusion as soon as $X,Y$ are infinite dimensional.

The opposite extreme to approximation numbers is attained by Hilbert numbers $$h_n(T)=\inf\{a_n(vTu): u\in \mathcal{L}(\ell^2(\mathbb{N}),X), v\in \mathcal{L}(Y,\ell^2(\mathbb{N})), \|u\|\|v\|\leq 1\}$$ 
since they are known to be smaller than any other $s$-numbers sequence: $h_n(T)\leq s_n(T)$ for all $T$ and $n$ \cite[p. 96]{Peigem}. Thus, it is not strange that there exist Banach spaces $X$, $Y$ such that the sequence $h_n(T)$ does not decrease slowly (see, e.g., \cite[Proposition 1.2]{O}, where 
$\lim_{n\to\infty}nh_n(T)=0$ is proved for any 
$T\in \mathcal{L}(c_0,\ell^1)$).

The following result gives a general condition on the spaces $X,Y$ to guarantee that every $s$-numbers sequence $s_n(T)$ decays slowly to zero for approximable operators $T:X\to Y$:

\begin{thm} \label{coco}
    Assume that $X,Y$ admit uniformly complemented copies of $\ell^2_n$ for each $n$. Then for every sequence $\{\varepsilon_n\} \in  c_0$ there exists an approximable operator $T\in\mathcal{L}^{\mathcal{A}}(X,Y)$ such that $s_n(T)\neq\mathbf{O}(\varepsilon_n)$
\end{thm}
\begin{proof}
Recall that a space $E$ is $c$-isomorphic to $\ell_n^2$ if there exists a linear isomorphism $u:E\to \ell_n^2$ such that $\max\{\|u\|,\|u^{-1}\|\}\leq c$. This property is usually denoted by $E\sim^c \ell_n^2$. A well known theorem by Dvoretzky states that for every natural number $n$ and every $\varepsilon>0$, there exists $N =N(n,\varepsilon)$ such that, if $Z$  is a Banach 
space of dimension $N$, there exists a subspace of $Z$ which is $(1+\varepsilon)$-isomorphic to $\ell_n^2$. We apply this theorem with $\varepsilon=1/2$.

Thus, we can assume that $P_n:X\to X$, $Q_n:Y\to Y$ are projections such that $\|P_n\|,\|Q_n\|\leq C$, $E_n=P_n(X)$, $F_n=Q_n(Y)$ satisfy $E_n,F_n\sim^{1+1/2}\ell^2_n$, since $X,Y$ admit uniformly complemented copies of $\ell^2_n$. Let $i_n:E_n\to X$, $j_n:F_n\to Y$ be the natural inclusions. Then $\widetilde{P}_n:X\to E_n$, $\widetilde{P}_n(x)=P_n(x)$ and $\widetilde{Q}_n:Y\to F_n$, $\widetilde{Q}_n(y)=Q_n(y)$ are well defined and satisfy $\|\widetilde{P}_n\|,\|\widetilde{Q}_n\|\leq C$. 
There are invertible operators $u_n:E_n\to \ell^2_n$ and $v_n:F_n\to \ell^2_n$ such that $\max\{\|u_n\|,\|u_n^{-1}\|,\|v_n\|,\|v_n^{-1}\|\}\leq \frac{3}{2}$. Thus
\begin{eqnarray*}
1\leq s_n(1_{\ell^2_n}) &=& s_n(v_n\circ v_n^{-1}\circ u_n\circ u_n^{-1})\\
&=& s_n(v_n\circ \rho_n\circ u_n^{-1})\leq \|v_n\|s_n( \rho_n)\| u_n^{-1}\|\leq \left(\frac{3}{2}\right)^2 s_n( \rho_n)
\end{eqnarray*}
where $\rho_n=v_n^{-1}\circ u_n:E_n\to F_n$. 
Let us now consider the operator $T_n=j_n\circ \rho_n\circ \widetilde{P}_n:X\to Y$. Then: $$\rho_n=\widetilde{Q}_n\circ T_n\circ i_n$$
Hence
\[
\frac{4}{9}\leq s_n(\rho_n)=s_n(\widetilde{Q}_n\circ T_n\circ i_n)\leq \|\widetilde{Q}_n\|s_n(T_n)\|i_n\|
\]
and 
\[
s_n(T_n)\geq \frac{4}{9C}
\]
Note that $\|T_n\|$ is uniformly bounded by a certain constant $K>0$, by construction. Hence $s_n(\frac{T_n}{\|T_n\|})\geq \frac{4}{9CK}$.  
\end{proof}

\begin{cor} \label{coco1}
Assume that $X,Y$ admit uniformly complemented copies of $\ell^2_n$ for each $n$. Then    $\mathcal{L}_{p,q}^{(s)}(X,Y) \hookrightarrow \mathcal{L}^{\mathcal{A}}(X,Y)$ with strict inclusion for all $0<p<q\leq\infty$ and for every $s$-numbers sequence $\{s_n\}$. 
\end{cor}
 It is well known  \cite[Theorem 19.3]{DJT} that every $B$-convex Banach space contains uniformly complemented $\ell_n^2$'s, so that we can apply Theorem \ref{coco} and Corollary \ref{coco1} to these spaces.

\section{Appendix: The $K$-interpolation method for the sum and the intersection in quasi-Banach setting}

In this appendix we demonstrate that formula \eqref{sumintersection} holds for quasi-Banach couples $(A_0,A_1)$ and $(\theta,p)\in ]0,1[\times ]0,\infty[$. The case of Banach couples is well-known (see e.g., \cite{AA,H,M0,M,P}). Unfortunately, we have been unable to find a reference for the quasi-Banach setting. Hence we include the proof in this appendix, just for the sake of completeness. Indeed, our proof follows all steps given in \cite{AA} (see also \cite{M0}), with small adaptations to include the quasi-Banach inequality of the norm, when necessary. In order to make the proof as clear as possible, we also add more details than those given in \cite{AA}. We have chosen the proof in \cite{AA} because it is a very general one, which fits with the spirit of this paper.
Thus, in all what follows, we work with interpolation spaces $(A_0,A_1)_{\Phi}=\{a\in A_0+A_1:\|K(a,t,A_0,A_1)\|_{\Phi}<\infty\}$ where $\Phi$ is a parameter space for the $K$-interpolation method, and $(A_0,A_1)$ is a quasi-Banach couple. We denote by $M_i$ the quasi-norm constant of the space $A_i$ (i.e, $\|a+b\|_{A_i}\leq M_i(\|a\|_{A_i}+\|b\|_{A_i}$), $i=0,1$.  
\begin{lem} \label{lemA1}
    If $A_1\hookrightarrow A_0$, then 
    \[
    \|x\|_{A_0+A_1}\asymp \|x\|_{A_0} \text{ \rm{ for } } x\in A_0.
    \]
\end{lem}
\begin{proof}
    By hypothesis, we have that $\|x\|_{A_0}\leq C\|x\|_{A_1}$ for all $x\in A_1$ and a certain constant $C\geq 1$. Now, given $x\in A_0$, we have that $\|x\|_{A_0+A_1}=\inf_{x=a+b}\|a\|_{A_0}+\|b\|_{A_1}\leq \|x\|_{A_0}$. Thus, we only need to demonstrate the reverse inequality. 
    By definition of $\inf$, there are $a\in A_0$, $b\in A_1$, such that $x=a+b$ and $\|a\|_{A_0}+\|b\|_{A_1}\leq 2\|x\|_{A_0+A_1}$. On the other hand, $b\in A_1$ implies that $\|b\|_{A_1}\geq \frac{1}{C}\|b\|_{A_0}$. Hence
    \begin{eqnarray*}
    \frac{1}{CM_0}\|x\|_{A_0} &=& \frac{1}{C}(\frac{1}{M_0}\|a+b\|_{A_0}) \leq \frac{1}{C}(\|a\|_{A_0}+\|b\|_{A_0})\\ &\leq& \frac{1}{C}\|a\|_{A_0}+\|b\|_{A_1}\\ 
    &\leq&  \|a\|_{A_0}+\|b\|_{A_1} \leq 2\|x\|_{A_0+A_1}.
    \end{eqnarray*}
\end{proof}
The following result is demonstrated in \cite{M0} for the Banach setting:
\begin{lem}\label{mainlema}
\[K(x,t,A_0+A_1,A_0\cap A_1)\asymp K(x,t,A_0,A_1)+tK(x,1/t,A_0,A_1) \text{ \rm{ for all }} 0<t<1
\]    
\end{lem}
\begin{proof}
    As a first step, we prove that 
    \begin{equation} \label{fund1}
       K(a,t,A_0+A_1,A_i)\asymp K(a,t,A_{1-i},A_i) \quad (a\in A_0+A_1,\ 0<t\leq 1,\ i=0,1) 
    \end{equation}
    We prove \eqref{fund1} for $i=0$ since the other case is similar. Take $a\in A_0+A_1$ and $0<t\leq 1$. There is $a_t\in A_0$ such that 
    \[
    \|a-a_t\|_{A_0+A_1}+t\|a_t\|_{A_0}<2K(a,t,A_0+A_1,A_0)
    \]
    Moreover, there exists a decomposition $a-a_t=a_{0t}+a_{1t}$ such that $a_{it}\in A_i$, $i=0,1$, and
    \[
\|a_{0t}\|_{A_0}+\|a_{0t}\|_{A_0}\leq 2\|a-a_t\|_{A_0+A_1}
    \]
    Hence $a=a_{0t}+a+a_{1t}$ satisfies
    \begin{eqnarray*}
        K(a,t,A_1,A_0) &\leq&  \|a_{1t}\|_{A_1}+t\|a_{0t}+a_{t}\|_{A_0} \\
        &\leq&  \|a_{1t}\|_{A_1}+tM_0(\|a_{0t}\|_{A_0}+\|a_{t}\|_{A_0}) \\
        &\leq&  M_0(\|a_{1t}\|_{A_1}+t\|a_{0t}\|_{A_0}+t\|a_{t}\|_{A_0}) \quad \text{ \rm{(since  }} M_0\geq 1 \text{ \rm{)}}\\
        &\leq&  M_0(\|a_{1t}\|_{A_1}+\|a_{0t}\|_{A_0}+t\|a_{t}\|_{A_0}) \quad \text{ \rm{(since  }} 0<t\leq 1 \text{ \rm{)}}\\
        &\leq&  M_0(2\|a-a_t\|_{A_0+A_1}+t\|a_{t}\|_{A_0})\\
        &\leq&  4M_0K(a,t,A_0+A_1,A_0)
    \end{eqnarray*}
     Thus,
    \[ K(a,t,A_1,A_0)\lesssim K(a,t,A_0+A_1,A_0) \quad (0<t\leq 1). 
    \]
    On the other hand, the inequality $K(a,t,A_1,A_0)\geq K(a,t,A_0+A_1,A_0)$ holds trivially. Thus, 
    \[
    K(a,t,A_1,A_0) \asymp K(a,t,A_0+A_1,A_0) \quad (0<t\leq 1). 
    \]
    Given $a\in A_0+A_1$ and $0<t\leq 1$, we have that
    \begin{eqnarray*}
&\ & \max(K(a,t,A_0,A_1),tK(a,1/t,A_0,A_1))\\
&\ & \quad \quad  = \max(K(a,t,A_1,A_0),K(a,t,A_0,A_1))\\
&\ & \quad \quad \asymp  \max(K(a,t,A_0+A_1,A_0),K(a,t,A_0+A_1,A_1))\\
&\ & \quad \quad \leq K(a,t,A_0+A_1,A_0\cap A_1)
    \end{eqnarray*}
Now, 
we can find decompositions $a=a_{0t}+a_{1t}=a_{0t}^*+a_{1t}^*$  with $a_{it},a_{it}^*\in A_i$, $i=0,1$, such that 
\[
\|a_{0t}\|_{A_0}+t\|a_{1t}\|_{A_1} \leq 2K(a,t,A_0,A_1)
\]
\[
\|a_{0t}^*\|_{A_0}+(1/t)\|a_{1t}^*\|_{A_1} \leq 2K(a,1/t,A_0,A_1)
\]
Then $a_t=a_{0t}-a_{0t}^*=a_{1t}-a_{1t}^*\in A_0\cap A_1$ and 
\begin{eqnarray*}
\|a-a_t\|_{A_0+A_1} &=& \|a_{0t}+a_{1t}^*\|_{A_0+A_1}\leq \|a_{0t}\|_{A_0}\|+a_{1t}^*\|_{A_1}\\
&\leq& 2K(a,t,A_0,A_1)+2tK(a,1/t,A_0,A_1) 
\end{eqnarray*}
Moreover, 
\begin{eqnarray*}
t\|a_t\|_{A_0\cap A_1}&=& t\max(\|a_{0t}-a_{0t}^*\|_{A_0},\|a_{1t}-a_{1t}^*\|_{A_1})\\
&\leq & \max(M_0,M_1)t\max(\|a_{0t}\|_{A_0}+\|a_{0t}^*\|_{A_0},\|a_{1t}\|_{A_1}+\|a_{1t}^*\|_{A_1})\\
&\leq& 2\max(M_0,M_1)t \max(K(a,1/t,A_0,A_1)+K(a,t,A_0,A_1),\\
&\ & \quad (1/t)K(a,t,A_0,A_1)+tK(a,1/t,A_0,A_1))\\
&\leq& 2\max(M_0,M_1)(K(a,t,A_0,A_1)+tK(a,1/t,A_0,A_1)) 
\end{eqnarray*}
since  $0<t\leq 1$. Hence 
\[
K(a,t,A_0+A_1,A_0\cap A_1)\lesssim K(a,t,A_0,A_1)+tK(a,1/t,A_0,A_1)  
\]
for $a\in A_0+A_1$ and $\ 0<t\leq 1$. This ends the proof. 
\end{proof}
\begin{lem} \label{lemA2}
    \[
    \|x\|_{(A_0,A_1)_{\Phi}} \asymp \|\chi_{(0,1)}(t)K(x,t,A_0,A_1)\|_{\Phi}+\|\chi_{(1,\infty)}(t)K(x,t,A_0,A_1)\|_{\Phi}. 
    \]
\end{lem}
\begin{proof}
    By definition, 
    \begin{eqnarray*}
    \|x\|_{(A_0,A_1)_{\Phi}} &=& \|K(x,t,A_0,A_1)\|_{\Phi} \\ 
    &=& \|\chi_{(0,1)}(t)K(x,t,A_0,A_1)+\chi_{(1,\infty)}(t)K(x,t,A_0,A_1)\|_{\Phi} \\
    &\leq & M(\|\chi_{(0,1)}(t)K(x,t,A_0,A_1)\|_{\Phi}+\|\chi_{(1,\infty)}(t)K(x,t,A_0,A_1)\|_{\Phi})
    \end{eqnarray*}
    which proves one of the inequalities. To prove the other, we take into account that $\Phi$ is monotone, so that $\|u\|_{\Phi}\geq \max\{\|\chi_{(0,1)}(t)u(t)\|_{\Phi}, \|\chi_{(1,\infty)}(t)u(t)\|_{\Phi}\}$, which implies that 
    \[
    \|u\|_{\Phi}\geq \frac{1}{2}(\|\chi_{(0,1)}(t)u(t)\|_{\Phi}+\|\chi_{(1,\infty)}(t)u(t)\|_{\Phi})
    \]
    for any non-negative function $u$. The result follows applying this formula with $u(t)=K(x,t,A_0,A_1)$. 
\end{proof}
\begin{cor}
    If the couple is not ordered, then for every parameter space $\Phi$ we have that either $(A_0,A_1)_{\Phi}$ or $(A_1,A_0)_{\Phi}$ (or both) are strictly contained into $A_0+A_1$.
\end{cor}
\begin{proof}
    If the couple is not ordered, then either $A_0+A_1\neq A_1$ or $A_0+A_1\neq A_0$. This implies that either $K(S(A_0+A_1),t,A_0,A_1)=1$ for all $t>0$ or $K(S(A_0+A_1),t,A_1,A_0)=1$ for all $t>0$, respectively. 
    Assume, for example, that $A_0+A_1\neq A_1$, so that $K(S(A_0+A_1),t,A_0,A_1)=1$ for all $t>0$ (the other case is similar). Then there are elements $a_n\in S(A_0+A_1)$ such that $K(a_n,\frac{1}{n},A_0,A_1)\geq \frac{1}{2}$, $n=1,2,\cdots.$ Hence $$K(a_n,t,A_0,A_1)\geq \frac{1}{2} \chi_{(1/n,1)}(t) \quad \text{ \rm{ for all }} t$$
    so that $$\|a_n\|_{(A_0,A_1)_{\Phi}}=\|K(a_n,t,A_0,A_1)\|_{\Phi} \geq \|\chi_{(1/n,1)}(t)\|_{\Phi}\to \infty$$
    which implies that $(A_0,A_1)_{\Phi}$ is strictly embedded into $A_0+A_1$. 
\end{proof}
\begin{lem}\label{lemA3}
    If $A_1\hookrightarrow A_0$ then 
    \[
\|x\|_{(A_0,A_1)_{\Phi}} \asymp \|\chi_{(0,1)}(t)K(x,t,A_0,A_1)\|_{\Phi} \text{ \rm{ for } } x\in A_0
    \]
\end{lem}
\begin{proof}
    We only need to prove that $\|x\|_{(A_0,A_1)_{\Phi}}\leq M\|\chi_{(0,1)}(t)K(x,t,A_0,A_1)\|_{\Phi}$ for a certain constant $M>0$ and all $x\in A_0$, since the other inequality is trivial. Now, Lemma \ref{lemA2} implies that 
    \begin{eqnarray*}
        \|x\|_{(A_0,A_1)_{\Phi}} &\leq& M(\|\chi_{(0,1)}(t)K(x,t,A_0,A_1)\|_{\Phi}+\|\chi_{(1,\infty)}(t)K(x,t,A_0,A_1)\|_{\Phi})\\
        &\leq& M(\|\chi_{(0,1)}(t)K(x,t,A_0,A_1)\|_{\Phi}+\|x\|_{A_0}\|\chi_{(1,\infty)}(t)\|_{\Phi})
    \end{eqnarray*}
    since $K(x,t,A_0,A_1)\leq \|x\|_{A_0}$ for all $t>0$. Moreover, $$\|t\cdot \chi_{(0,1)}(t)\|_{\Phi}, \|\chi_{(1,\infty)(t)}\|_{\Phi}<\infty$$ since $\min(1,t)\in \Phi$ and $\Phi$ is monotone. Finally, we can apply Lemma \ref{lemA1} to claim that $\|x\|_{A_0}\leq C\|x\|_{A_0+A_1}$. Hence the computation above continues as follows:
    \begin{eqnarray*}
        &\leq & M'(\|\chi_{(0,1)}(t)K(x,t,A_0,A_1)\|_{\Phi}+\|x\|_{A_0+A_1})
    \end{eqnarray*}
    Finally, as $\frac{K(x,t,A_0,A_1)}{t}$ is non-increasing, $K(x,t,A_0,A_1)\geq t K(x,1,A_0,A_1)=t\|x\|_{A_0+A_1}$ for all $0<t<1$. Thus
    $\|\chi_{(0,1)}(t)K(x,t,A_0,A_1)\|_{\Phi}\geq \|x\|_{A_0+A_1} \|t\cdot\chi_{(0,1)}(t)\|_{\Phi}$. In other words,
    \begin{equation}\label{recurso}
   \begin{array}{llll} \|x\|_{A_0+A_1} &\leq &  \frac{1}{\|t\cdot\chi_{(0,1)}(t)\|_{\Phi}}\|\chi_{(0,1)}(t)K(x,t,A_0,A_1)\|_{\Phi} \\ 
   & = & C\|\chi_{(0,1)}(t)K(x,t,A_0,A_1)\|_{\Phi}\end{array}
    \end{equation}
    Hence the chain of inequalities we are computing continues as:
    \begin{eqnarray*}
        \|x\|_{(A_0,A_1)_{\Phi}} &\leq& M'(\|\chi_{(0,1)}(t)K(x,t,A_0,A_1)\|_{\Phi}+C\|\chi_{(0,1)}(t)K(x,t,A_0,A_1)\|_{\Phi})\\
        &=& (M'+M'C)K(x,t,A_0,A_1)\|_{\Phi},
    \end{eqnarray*}
    which ends the proof. 
\end{proof}

For the following two theorems, we assume that the parameter spaces $\Phi_0,\Phi_1$ satisfy, for arbitrary functions $u$, the following inequalities:
\begin{equation}\label{Cond1}
    \|\chi_{(0,1)}(t)u(t)\|_{\Phi_0} \leq C_1\|\chi_{(0,1)}(t)u(t)\|_{\Phi_1}
\end{equation}
\begin{equation} \label{Cond2}
    \|\chi_{(1,\infty)}(t)u(t)\|_{\Phi_1} \leq C_2\|\chi_{(1,\infty)}(t)u(t)\|_{\Phi_0}
\end{equation}
\begin{thm}\label{InterSum1}
Assume that the parameter spaces $\Phi_0,\Phi_1$ satisfy \eqref{Cond1},\eqref{Cond2}, and
\begin{equation} \label{Cond3}
    \|\chi_{(1,\infty)}(t)u(t)\|_{\Phi_0} \asymp \|\chi_{(0,1)}(t)tu(1/t)\|_{\Phi_1}
\end{equation}
Then
    \[(A_0,A_1)_{\Phi_0}\cap (A_0,A_1)_{\Phi_1}=(A_0+A_1,A_0\cap A_1)_{\Phi_1}\]
\end{thm}
\begin{proof}
    Let us set $B_0=(A_0,A_1)_{\Phi_0}$, $B_1=(A_0,A_1)_{\Phi_1}$, and $B=(A_0+A_1,A_0\cap A_1)_{\Phi_1}$. Let $x\in A_0+A_1$. Then Lemma \ref{lemA2} implies that
    \begin{eqnarray*}
    &\ & \|x\|_{B_0}+\|x\|_{B_1} \\ 
    &\asymp& \|\chi_{(0,1)}(t)K(x,t,A_0,A_1)\|_{\Phi_0}+\|\chi_{(1,\infty)}(t)K(x,t,A_0,A_1)\|_{\Phi_0} \\
        &\ & \quad + \|\chi_{(0,1)}(t)K(x,t,A_0,A_1)\|_{\Phi_1}+\|\chi_{(1,\infty)}(t)K(x,t,A_0,A_1)\|_{\Phi_1}
    \end{eqnarray*}
    Now, using \eqref{Cond1} and \eqref{Cond2} we get that 
    \begin{eqnarray*}
       &\ &   \|x\|_{B_0}+\|x\|_{B_1} \\ 
       &\asymp& \|\chi_{(0,1)}(t)K(x,t,A_0,A_1)\|_{\Phi_1}+\|\chi_{(1,\infty)}(t)K(x,t,A_0,A_1)\|_{\Phi_0} 
    \end{eqnarray*}
    Now, \eqref{Cond3} implies that
    \[
    \|\chi_{(1,\infty)}(t)K(x,t,A_0,A_1)\|_{\Phi_0}\asymp \|\chi_{(0,1)}(t)tK(x,1/t,A_0,A_1)\|_{\Phi_1}
    \]
    and, applying Lemmas \ref{mainlema} and \ref{lemA3}, as well as the fact that $\Phi_1$ is a parameter space, 
    so that, if $0\leq u(t),v(t)$ everywhere, then 
    \begin{equation}\label{arg}
\frac{\|u\|_{\Phi_1}+\|v\|_{\Phi_1}}{2} \leq \max(\|u\|_{\Phi_1},\|v\|_{\Phi_1})\leq  \|u+v\|_{\Phi_1}\leq M(\|u\|_{\Phi_1}+\|v\|_{\Phi_1}),
    \end{equation}
    and  
    we get 
    \begin{eqnarray*}
   &\ &  \|x\|_{B_0}+\|x\|_{B_1} \\ 
   &\asymp& \|\chi_{(0,1)}(t)K(x,t,A_0,A_1)\|_{\Phi_1}+\|\chi_{(0,1)}(t)tK(x,1/t,A_0,A_1)\|_{\Phi_1} \\
     &\asymp& \|\chi_{(0,1)}(t)K(x,t,A_0,A_1)+\chi_{(0,1)}(t)tK(x,1/t,A_0,A_1)\|_{\Phi_1} \\
     &\asymp& \|\chi_{(0,1)}(t)K(x,t,A_0+A_1,A_0\cap A_1)\|_{\Phi_1}\\
     &\asymp& \|x\|_{B}.
    \end{eqnarray*}
\end{proof}

\begin{thm}\label{InterSum0}
Assume that the parameter spaces $\Phi_0,\Phi_1$ satisfy \eqref{Cond1},\eqref{Cond2}, and 
\begin{equation} \label{Cond4}
    \|\chi_{(1,\infty)}(t)g(t)\|_{\Phi_1} \asymp \|\chi_{(0,1)}(t)tg(1/t)\|_{\Phi_0}
\end{equation}
Then
    \[(A_0,A_1)_{\Phi_0}+ (A_0,A_1)_{\Phi_1}=(A_0+A_1,A_0\cap A_1)_{\Phi_0}\]
\end{thm}
\begin{proof}
    Let us set $B_0=(A_0,A_1)_{\Phi_0}$, $B_1=(A_0,A_1)_{\Phi_1}$, and $B=(A_0+A_1,A_0\cap A_1)_{\Phi_0}$. Let $x\in B_0+B_1$, $x=x_0+x_1$, $x_i\in B_i$, $i=0,1$. Then Lemma \ref{mainlema}, in combination with the argument given in \eqref{arg}, which holds for every parameter space, gives
    \[
    \|x\|_{B}\asymp \|\chi_{(0,1)}(t)K(x,t,A_0,A_1)\|_{\Phi_0}+\|\chi_{(0,1)}(t)tK(x,1/t,A_0,A_1)\|_{\Phi_0}
    \]
    Now we use \eqref{Cond4} to get
     \begin{eqnarray*}
         \|x\|_{B} &\asymp&  \|\chi_{(0,1)}(t)K(x,t,A_0,A_1)\|_{\Phi_0}+\|\chi_{(1,\infty)}(t)K(x,t,A_0,A_1)\|_{\Phi_1} \\
         &\lesssim & \|\chi_{(0,1)}(t)K(x_0,t,A_0,A_1)\|_{\Phi_0}+\|\chi_{(1,\infty)}(t)K(x_0,t,A_0,A_1)\|_{\Phi_1}\\
         &\ & \quad +\|\chi_{(0,1)}(t)K(x_1,t,A_0,A_1)\|_{\Phi_0}+\|\chi_{(1,\infty)}(t)K(x_1,t,A_0,A_1)\|_{\Phi_1} 
     \end{eqnarray*}
and, by \eqref{Cond1} and \eqref{Cond2}, we get 
\begin{eqnarray*}
    \|x\|_{B} &\lesssim& \|\chi_{(0,1)}(t)K(x_0,t,A_0,A_1)\|_{\Phi_0}+\|\chi_{(1,\infty)}(t)K(x_0,t,A_0,A_1)\|_{\Phi_0}\\
         &\ & \quad +\|\chi_{(0,1)}(t)K(x_1,t,A_0,A_1)\|_{\Phi_1}+\|\chi_{(1,\infty)}(t)K(x_1,t,A_0,A_1)\|_{\Phi_1} \\ 
         &\asymp& \|x_0\|_{B_0}+\|x_1\|_{B_1}
\end{eqnarray*}
    It follows that $\|x\|_{B}\lesssim \|x\|_{B_0+B_1}$ since the inequality above holds true for every decomposition $x=x_0+x_1$. 

    Let us now demonstrate the inequality $\|x\|_{B_0+B_1}\lesssim \|x\|_{B}$. Let $x\in B$ and take $x=a_0+a_1$, with $a_i\in A_i$, $i=0,1$, such that $\|a_0\|_{A_0}+\|a_1\|_{A_1}\leq 2\|x\|_{A_0+A_1}$. Then
    \begin{eqnarray*}
        \|a_0\|_{B_0} &\asymp& \|\chi_{(0,1)}(t)K(a_0,t,A_0,A_1)\|_{\Phi_0}+\|\chi_{(1,\infty)}(t)K(a_0,t,A_0,A_1)\|_{\Phi_0} \\
        &\lesssim& \|\chi_{(0,1)}(t)K(a_0,t,A_0,A_1)\|_{\Phi_0}+\|\chi_{(1,\infty)}(t)\|_{\Phi_0}\|a_0\|_{A_0} \\
        &\ &\quad \quad \quad \text{ \rm{(since } } K(a_0,t,A_0,A_1)\leq \|a_0\|_{A_0} \text{ \rm{)}} \\
        &\asymp& \|\chi_{(0,1)}(t)K(a_0,t,A_0,A_1)\|_{\Phi_0}+\|a_0\|_{A_0}.
    \end{eqnarray*}
    Now, $a_0=x-a_1$ implies that 
    \begin{eqnarray*}
       &\ &  \|a_0\|_{B_0} \\
       &\lesssim&  \|\chi_{(0,1)}(t)K(x,t,A_0,A_1)\|_{\Phi_0}+\|\chi_{(0,1)}(t)K(a_1,t,A_0,A_1)\|_{\Phi_0}+\|a_0\|_{A_0} \\
        &\lesssim&  \|\chi_{(0,1)}(t)K(x,t,A_0,A_1)\|_{\Phi_0}+\|\chi_{(0,1)}(t)t\|_{\Phi_0}\|a_1\|_{A_1}+\|a_0\|_{A_0} \\
        &\asymp&  \|\chi_{(0,1)}(t)K(x,t,A_0,A_1)\|_{\Phi_0}+\|a_1\|_{A_1}+\|a_0\|_{A_0}\\
        &\leq&  \|\chi_{(0,1)}(t)K(x,t,A_0,A_1)\|_{\Phi_0}+2\|x\|_{A_0+A_1}\\
        &\lesssim&  \|\chi_{(0,1)}(t)K(x,t,A_0,A_1)\|_{\Phi_0} \text{ \rm{ (by }} \eqref{recurso} \text{\rm{)}}
    \end{eqnarray*}
    Thus 
    \[
    \|a_0\|_{B_0} \lesssim \|\chi_{(0,1)}(t)K(x,t,A_0,A_1)\|_{\Phi_0}, 
    \]
    and a similar argument gives:
    \[
    \|a_1\|_{B_1} \lesssim \|\chi_{(1,\infty)}(t)K(x,A_0,A_1)\|_{\Phi_1}
    \]
    Using \eqref{Cond4}, we get 
    \[
    \|a_1\|_{B_1} \lesssim \|\chi_{(1,\infty)}(t)tK(x,1/t,A_0,A_1)\|_{\Phi_0}
    \]
     Hence
     \begin{eqnarray*}
      &\ &   \|a_0\|_{B_0}+\|a_1\|_{B_1} \\
      &\lesssim & \|\chi_{(0,1)}(t)K(x,t,A_0,A_1)\|_{\Phi_0}+\|\chi_{(1,\infty)}(t)tK(x,1/t,A_0,A_1)\|_{\Phi_0} \\
         &\asymp & \|\chi_{(0,1)}(t)K(x,t,A_0,A_1)+\chi_{(1,\infty)}(t)tK(x,1/t,A_0,A_1)\|_{\Phi_0}  \text{ \rm{ (by }} \eqref{arg} \text{\rm{)}}\\
         &\asymp& \|x\|_{B} \text{ \rm{ (by Lemma}} \eqref{mainlema} \text{\rm{)}}
     \end{eqnarray*}
     Hence $\|x\|_{B_0+B_1}\leq \|x\|_B$. This ends the proof.     
\end{proof}

\begin{thm} \label{sumainterseccion}
    Let $(A_0,A_1)$ be a quasi-Banach couple and let  $(\theta,p)\in ]0,1[\times ]0,\infty[$. Then 
\begin{equation}\label{sumintersection1}
    (A_0+A_1,A_0\cap A_1)_{\theta,p}=\left\{\begin{array}{cccc}
   (A_0,A_1)_{\theta,p}+(A_0,A_1)_{1-\theta,p} & 0\leq \theta \leq 1/2 \\
     (A_0,A_1)_{\theta,p}\cap (A_0,A_1)_{1-\theta,p} & 1/2\leq \theta \leq 1 
\end{array}\right.
\end{equation}
\end{thm}
\begin{proof} This is an easy corollary to Theorems \ref{InterSum1} and \ref{InterSum0}. Indeed, if we set $$\Phi_0=\Phi(\theta)=\{u: \|u\|_{\Phi(\theta)}=\left(\int_0^{\infty}\frac{|u(t)|^p}{t^{\theta}}\frac{dt}{t}\right)^{1/p}<\infty\}$$ and  $\Phi_1=\Phi(1-\theta)$, then  $\Phi_0,\Phi_1$ satisfy the hypotheses of Theorem \ref{InterSum0} for $0<\theta<1/2$, $0<p<\infty$ and the hypotheses of Theorem \ref{InterSum1} for $1/2\leq \theta <1$, $0<p<\infty$. 
\end{proof}

\section{Acknowledgement}
We want to express our thanks to P. Nilsson who read the manuscript and made several interesting suggestions. We also want to express our warmest thanks to an anonymous referee who detected an error in the proof of one theorem of the first version of the paper. With his/her help, the paper has improved its quality and its readability.

\end{document}